\RequirePackage{fixltx2e}
\documentclass[12pt,english]{article}
\usepackage[T1]{fontenc}
\usepackage[utf8]{inputenc}
\usepackage{verbatim}
\usepackage{prettyref}
\usepackage{mathrsfs}
\usepackage{mathtools}
\usepackage{amsmath}
\usepackage{amsthm}
\usepackage{amssymb}
\usepackage{stmaryrd}
\usepackage[numbers]{natbib}

\makeatletter
\numberwithin{figure}{section}
\theoremstyle{plain}
\newtheorem{thm}{\protect\theoremname}[section]
\theoremstyle{definition}
\newtheorem{defn}[thm]{\protect\definitionname}
\ifx\proof\undefined
\newenvironment{proof}[1][\protect\proofname]{\par
  \normalfont\topsep6\p@\@plus6\p@\relax
  \trivlist
  \itemindent\parindent
\item[\hskip\labelsep\scshape #1]\ignorespaces
}{%
  \endtrivlist\@endpefalse
}
\providecommand{\proofname}{Proof}
\fi
\theoremstyle{plain}
\newtheorem{cor}[thm]{\protect\corollaryname}
\theoremstyle{plain}
\newtheorem{lem}[thm]{\protect\lemmaname}
\theoremstyle{remark}
\newtheorem{rem}[thm]{\protect\remarkname}
\theoremstyle{definition}
\newtheorem{example}[thm]{\protect\examplename}

\pdfoutput=1

\usepackage{lmodern}
\usepackage{microtype}

\usepackage[margin=1in]{geometry}

\usepackage[T1]{fontenc}

\usepackage[dvipsnames]{xcolor}
\newcommand\myshade{100}
\definecolor{mylinkcolorhtml}{HTML}{0066cc}
\definecolor{mycitecolorhtml}{HTML}{008800}
\definecolor{myurlcolorhtml}{HTML}{0066cc}
\colorlet{mylinkcolor}{mylinkcolorhtml}
\colorlet{mycitecolor}{mycitecolorhtml}
\colorlet{myurlcolor}{myurlcolorhtml}
\usepackage{hyperref}
\hypersetup{
  linkcolor  = mylinkcolor!\myshade!black,
  citecolor  = mycitecolor!\myshade!black,
  urlcolor   = myurlcolor!\myshade!black,
  colorlinks = true,
}
\providecommand{\doi}{}
\renewcommand{\doi}[1]{\href{https://doi.org/#1}{doi:
\begingroup\urlstyle{rm}\nolinkurl{#1}\endgroup}}

\renewcommand*{\leq}{\leqslant}
\renewcommand*{\geq}{\geqslant}

\usepackage{amsthm}
\numberwithin{equation}{section}
\numberwithin{thm}{section}
\theoremstyle{remark}
\newtheorem*{rem*}{Remark}
\newtheorem*{notation*}{Notation}
\theoremstyle{definition}
\newtheorem{assumption}[thm]{Assumption}
\newtheorem*{ack*}{Acknowledgements}

\usepackage{enumitem}

\usepackage{upref}

\usepackage{tikz}

\newrefformat{assu}{Assumption \ref{#1}}
\newrefformat{prop}{Proposition \ref{#1}}
\newrefformat{fig}{Figure \ref{#1}}
\newrefformat{tab}{Table \ref{#1}}
\newrefformat{def}{Definition \ref{#1}}
\newrefformat{exa}{Example \ref{#1}}
\newrefformat{cor}{Corollary \ref{#1}}
\newrefformat{sec}{\S\ref{#1}}
\newrefformat{sub}{\S\ref{#1}}
\newrefformat{subsec}{\S\ref{#1}}
\newrefformat{app}{Appendix \ref{#1}}
\newrefformat{rem}{Remark \ref{#1}}
\newrefformat{line}{Line \ref{#1}}

\makeatother

\usepackage{babel}
\providecommand{\corollaryname}{Corollary}
\providecommand{\definitionname}{Definition}
\providecommand{\examplename}{Example}
\providecommand{\lemmaname}{Lemma}
\providecommand{\remarkname}{Remark}
\providecommand{\theoremname}{Theorem}

\begin{document}
\global\long\def\leq{\leqslant}
\global\long\def\geq{\geqslant}
\date{}

\title{A zero-sum stochastic differential game with impulses, precommitment,
and unrestricted cost functions }

\author{Parsiad Azimzadeh\thanks{David R. Cheriton School of Computer
    Science, University of Waterloo,
    Waterloo ON, Canada N2L 3G1 {\tt
\href{mailto:pazimzad@uwaterloo.ca}{pazimzad@uwaterloo.ca}}.} }
\maketitle
\begin{abstract}
  We study a zero-sum stochastic differential game (SDG) in which
  one controller plays an impulse control while their opponent plays
  a stochastic control. We consider an asymmetric setting in which the
  impulse player commits to, at the start of the game, performing less
  than $q$ impulses ($q$ can be chosen arbitrarily large). In order
  to obtain the uniform continuity of the value functions, previous
  works involving SDGs with impulses assume the cost of an impulse to
  be decreasing in time. Our work avoids such restrictions by using
  a dynamic programming principle with finite-valued stopping times.
  We establish that the resulting
  game admits a value, and in turn, the existence and uniqueness of
  viscosity solutions to an associated Hamilton-Jacobi-Bellman-Isaacs
  quasi-variational inequality.

\end{abstract}
\begin{description}
  \item [{AMS~subject~classifications.}] 49L20, 49L25, 91A23, 91A15
  \item [{Keywords.}] zero-sum stochastic differential game, impulse control,
    quasi-variational inequalities, viscosity solutions
\end{description}

\section{\label{sec:introduction}Introduction}

\subsection{Context and literature}

The theory of \emph{stochastic differential games} (SDGs) can be traced
back to the introduction of deterministic differential games (DDGs)
by Isaacs \citep{MR0210469}. Using the notion of Elliott-Kalton strategies
\citep{MR0359845}, Evans and Souganidis considered DDGs using viscosity
theory \citep{MR756158}. This was followed by the pioneering work
of Fleming and Souganidis, who also used viscosity theory to consider
SDGs \citep{MR997385}. Since then, SDGs have been studied under different
settings (e.g. zero-sum or nonzero-sum games, nonlinear cost functions,
asymmetric information, etc.) and using various tools (e.g., backward
  stochastic differential equations, path dependent partial differential
equations, stochastic Perron's method, etc.). We list a few such works
here:
\citep{MR1321134,MR1752678,MR2086176,MR2373477,MR2465706,MR2757742,MR3206980,MR3162260,MR3227460,MR3375882,MR3574709}.

The references listed above are mainly concerned with SDGs under stochastic
controls, in which the controls of both players influence the drift
and diffusion of a stochastic differential equation (SDE). We consider
instead an \emph{impulse control} problem, in which the actions of
a player affect the system in an ``instantaneous'' manner. We are
aware of only a few works \citep{MR2826978,MR3053571} that study
\emph{zero-sum SDGs with impulse control} (along with some related
studies on switching controls; e.g., \citep{MR1306930,MR2338432,MR3553921}).
This is in spite of the fact that impulse control problems have enjoyed
a resurgence (see, e.g., \citep{MR3070528,MR3071398}) due to a demand
for more realistic financial models (e.g., fixed transaction costs
and liquidity risk) \citep{MR2284012,MR2513141,MR2568293,MR3464413}
and their link to backward stochastic differential equations \citep{MR2642892}.

Our closest related works are \citep{MR2826978,MR3053571}.   \citep{MR2826978}
considers an infinite horizon game in which one player plays an impulse
control while their opponent plays a stochastic control. Their setting
is most similar to ours, save that our game is posed on a finite horizon.
The setting of \citep{MR3053571}, on the other hand, is one in which
both players play impulse controls.

The first distinguishing characteristic of our work is that we consider
an asymmetric setting in which the impulse player commits to, at the
start of the game, performing less than $q$ impulses ($q$ can be
chosen arbitrarily large). Ultimately, our analysis shows that this
assumption is not reflected in the resulting Hamilton-Jacobi-Bellman-Isaacs
quasi-variational inequality (HJBI-QVI) obtained from dynamic programming.
In other words, precommitment does not affect the value functions
(i.e., the game is ``robust with respect to precommitment'').

In order to obtain the uniform continuity of the value functions,
it is customarily assumed that the cost of performing an impulse decreases
with respect to time (see \citep[Assumption (H2)]{MR2826978} and
\citep[Eq. (2.6)]{MR3053571}). We are able to replace this assumption
by using a dynamic programming principle whose stopping times take
only finitely many values. However, this setting introduces various
nonstandard challenges in the dynamic programming arguments. For example,
we require a dynamic programming principle (DPP) that holds only under
``strongly nonanticipative families of stopping times''
(\prettyref{def:strongly_nonanticipative_family_of_stopping_times}),
introduced with the purpose of formalizing the intuitive notion that
the decision to stop should not depend on future information from
the controls. The dynamic programming principle
in this work results in further challenges in establishing that the
HJBI-QVI is the dynamic programming equation (DPE) associated with
the game.

For completeness, we also mention the works
\citep{MR1264011,MR2167312,MR2166444,MR2243838,MR2341237,aid2016nonzero},
which study impulse control games in various other settings (e.g.,
as DDGs, as nonzero-sum games, etc.). We do not claim this list to
be exhaustive.

In \prettyref{sec:framework}, we establish our framework and list
our results, culminating in the value of an impulse control game and
an existence and uniqueness result for the associated HJBI-QVI.
\prettyref{sec:regularity}
establishes the regularity of the upper and lower value functions
and gives existence and uniqueness results for the impulse controlled
SDE. \S\ref{sec:dpp_proof}\textendash \ref{sec:comparison_principle_proof}
gather the proofs of the results listed in \prettyref{sec:framework}.
\prettyref{sec:conclusion} discusses extensions of the model.

\subsection{\label{subsec:contribution}Setting of our game}

In our game, two players compete on a finite horizon $[t,T]$ by influencing
a stochastic process, denoted $X$. The ``sup-player'' aims to maximize
a particular function, while the ``inf-player'' aims to minimize
it.

The sup-player exerts their control by choosing impulse times
$\tau_{1}\leq\tau_{2}\leq\cdots$
and impulse controls $z_{1},z_{2},\ldots$, denoted
$a\coloneqq(\tau_{j},z_{j})_{j}$
for brevity. The inf-player exerts their control by choosing a process
$(b_{t})_{t}$. Letting $(W_{t})_{t}$ denote a standard Brownian
motion, between impulse times, $X$ follows the stochastic differential
equation (SDE)
\[
  dX_{s}=\mu(X_{s},b_{s})ds+\sigma(X_{s},b_{s})dW_{s}.
\]
At an impulse time $\tau_{j}$, the process changes instantaneously
as a function of the corresponding impulse control $z_{j}$:
\[
  X_{\tau_{j}}=X_{\tau_{j}-}+\Gamma(\tau_{j},z_{j})
\]
where $t-$ is shorthand for a limit from the left.

As is usually the case in SDGs, players play not controls $(a,b)$
but rather nonanticipative strategies $(\alpha,\beta)$. Given a functional
$J\coloneqq J(t,x;a,b)$ whose first two arguments describe the initial
time and state of the process $X$, the upper and lower values of
our game are
\[
  \sup_{q\geq1}\adjustlimits\inf_{\beta}\sup_{a}J(t,x;a,\beta(a))\text{
  and }\sup_{q\geq1}\adjustlimits\sup_{\alpha}\inf_{b}J(t,x;\alpha(b),b)
\]
where the integer $q$ specifies the maximum number of allowed impulses
in the control $a$ (resp. strategy $\alpha$). When the upper and
lower values coincide, we say that the game admits a \emph{value}.
The asymmetry\footnote{By asymmetry, we mean that the upper value
  function is not defined
as $\inf_{\beta}\sup_{q\geq1}\sup_{a}J(t,x;a,\beta(a))$.} of the
value functions corresponds to the impulse player's precommitment.

\section{\label{sec:framework}Framework and statement of results}

Fix $T\in[0,\infty)$ and an $\mathbb{R}^{d_{W}}$-valued standard
Brownian motion $(W_{t})_{t\in[0,T]}$ on the canonical Wiener space
$(\Omega,\mathscr{F},\mathbb{P})$. Let $\Omega_{t,T}$ be the set
of continuous functions from $[t,T]$ to $\mathbb{R}^{d_{W}}$ starting
at zero and $\mathbb{P}_{t,T}$ its associated Wiener measure (i.e.,
$\Omega=\Omega_{0,T}$ and $\mathbb{P}=\mathbb{P}_{0,T}$). We omit
the subscripts on $\Omega$ and $\mathbb{P}$ (and the associated
expectation $\mathbb{E}$) whenever it is unambiguous to do so. We
denote by $\mathscr{F}_{t,s}$ the $\sigma$-algebra generated by
$(W_{u}-W_{t})_{u\in[t,s]}$ and augmented by all $\mathbb{P}$ null
sets. For an arbitrary subset $\mathcal{I}$ of $[t,\infty]$, we
denote by $\mathscr{T}_{t}(\mathcal{I})$ the set of all
$(\mathscr{F}_{t,s})_{s\in[t,T]}$-stopping
times $\tau$ such that $\operatorname{range}(\tau)\subset\mathcal{I}$.
\begin{defn}[Controls]
  \label{def:controls}A $[t,T]$ impulse control\emph{ }is a tuple
  $a\coloneqq(\tau_{j},z_{j})_{j\geq1}$ where
  $\tau_{j}\in\mathscr{T}_{t}([t,T]\cup\{\infty\})$
  for each $j$, each $z_{j}$ is an $\mathscr{F}_{t,\tau_{j}\wedge T}$-measurable
  random variable taking values in some Borel set $Z\subset\mathbb{R}^{d_{Z}}$,
  and $\tau_{1}\leq\tau_{2}\leq\cdots\leq\infty$. The set of all such
  controls is denoted $\mathcal{A}(t)$.

  A $[t,T]$ stochastic control is an
  $(\mathscr{F}_{t,s})_{s\in[t,T]}$-progressively
  measurable process $b\coloneqq(b_{s})_{s\in[t,T]}$ taking values
  in some Borel set $B\subset\mathbb{R}^{d_{B}}$. The set of all such
  controls is denoted $\mathcal{B}(t)$.
\end{defn}
Given controls $a$ and $b$ as above, the relevant SDE (with impulses)
is
\begin{equation}
  X_{s}=x+\int_{t}^{s}\mu(X_{u},b_{u})du+\int_{t}^{s}\sigma(X_{u},b_{u})dW_{u}+\sum_{\tau_{j}\leq
  s}\Gamma(\tau_{j},z_{j})\text{ for }s\in[t,T].\label{eq:sde}
\end{equation}
If it exists and is unique, we use $X^{t,x;a,b}$ to denote a solution
(see \prettyref{def:admissible_control}) to \eqref{eq:sde}. The
gain (resp. cost) functional for the sup(resp. inf)-player is given
by
\[
  J(t,x;a,b)\coloneqq\mathbb{E}\left[\int_{t}^{T}f(s,X_{s},b_{s})ds+\sum_{\tau_{j}\leq
  T}K(\tau_{j},z_{j})+g(X_{T})\right]
\]
where it is understood that $X\coloneqq X^{t,x;a,b}$. It is convenient
at this point to also define the intervention operator $\mathcal{M}$,
which (roughly speaking) describes the value of the game immediately
after an optimal impulse:
\begin{equation}
  \mathcal{M}u(t,x)\coloneqq\sup_{z\in Z}\left\{
  u(t,x+\Gamma(t,z))+K(t,z)\right\} .\label{eq:intervention}
\end{equation}

We are now ready to introduce admissible controls and strategies.
\begin{defn}[Admissible impulse control]
  \label{def:admissible_control}A $[t,T]$ impulse control $a\in\mathcal{A}(t)$
  is admissible at $x\in\mathbb{R}^{d}$ if for each $b\in\mathcal{B}(t)$,
  a solution of \eqref{eq:sde} exists and is unique. The set of all
  such controls is denoted $\mathcal{A}(t,x)$.

  By a \emph{solution} $X\coloneqq X^{t,x;a,b}$, we mean that $X$
  is $(\mathscr{F}_{t,s})_{s\in[t,T]}$-adapted, has càdlàg paths, is
  in $\mathbb{L}^{2}(\Omega_{t,T}\times[t,T])$, and satisfies \eqref{eq:sde}.
  Uniqueness is determined up to indistinguishability.
\end{defn}
We now introduce certain subsets of $\mathcal{A}(t,x)$ and $\mathcal{A}(t)$
which are used in our analysis. Below, for an impulse control
$a\coloneqq(\tau_{j},z_{j})_{j\geq1}\in\mathcal{A}(t)$,
we use
\[
  (\#a)_{s}\coloneqq0\vee\sup\{j\geq1\colon\tau_{j}\leq s\}
\]
to denote its total number of impulses on $[t,s]$ (recall that
$\sup\emptyset=-\infty$).
\begin{defn}[Impulse control subsets]
  For each integer $q\geq1$, let $\mathcal{A}^{q}(t,x)$ be the set
  of all impulse controls $a\in\mathcal{A}(t,x)$ such that
  \[
    \mathbb{P}(S(a))=1\text{ where }S(a)\coloneqq\{(\#a)_{T}<q\}.
  \]
  For each integer $q\geq1$ and Borel set $Q\subset Z$,
  let $\mathcal{A}^{q,Q}(t)$ be the set of all impulse controls
  $a\coloneqq(\tau_{j},z_{j})_{j\geq1}\in\mathcal{A}(t)$
  such that
  \[
    \mathbb{P}(S(a)\cap\left(\cap_{j\geq1}S_{j}(a)\right))=1\text{
    where }S_{j}(a)\coloneqq\left\{ (\#a)_{T}<j\right\} \cup\left\{
    z_{j}\in Q\right\} .
  \]
\end{defn}
Intuitively, $\mathcal{A}^{q}(t,x)$ is the set of \emph{admissible}
impulse controls with less than $q$ impulses. Similarly, $\mathcal{A}^{q,Q}(t)$
is the set of impulse controls with less than $q$ impulses and with
each impulse contained in the set $Q$.

Let $\bar{t},t\in[0,T]$ with $\bar{t}\leq t$. Given
$\omega\in\Omega_{\bar{t},T}$,
we define $(\omega_{1},\omega_{2})$ by
\begin{align*}
  \omega_{1} & \coloneqq\omega|_{[\bar{t},t]}\\
  \text{and }\omega_{2} & \coloneqq\left(\omega-\omega(t)\right)|_{[t,T]},
\end{align*}
identifying $\Omega_{\bar{t},T}$ with $\Omega_{\bar{t},t}\times\Omega_{t,T}$
along with
$\mathbb{P}_{\bar{t},T}=\mathbb{P}_{\bar{t},t}\otimes\mathbb{P}_{t,T}$.
Fix an arbitrary $z_{0}\in Z$. We notice that
\begin{itemize}
  \item for $b\in\mathcal{B}(\bar{t})$, the control $b|_{[t,T]}(\omega_{1})$
    defined by $(b|_{[t,T]}(\omega_{1}))(\omega_{2})_{s}\coloneqq b(\omega)_{s}$
    is a member of $\mathcal{B}(t)$ for $\mathbb{P}_{\bar{t},t}$-almost
    all $\omega_{1}$;
  \item for $a\coloneqq(\tau_{j},z_{j})_{j\geq1}\in\mathcal{A}(\bar{t})$,
    set $N\coloneqq(\#a)_{t}$. The control
    $a|_{(t,T]}(\omega_{1})\coloneqq(\hat{\tau}_{j}(\omega_{1}),\hat{z}_{j}(\omega_{1}))_{j\geq1}$
    defined by
    $(\hat{\tau}_{j}(\omega_{1}))(\omega_{2})\coloneqq\tau_{N(\omega_{1})+j}(\omega)$
    and
    $(\hat{z}_{j}(\omega_{1}))(\omega_{2})\coloneqq z_{N(\omega_{1})+j}(\omega)$
    when $N(\omega_{1})<\infty$, and by
    $(\hat{\tau}_{j}(\omega_{1}))(\omega_{2})\coloneqq\infty$ and
    $(\hat{z}_{j}(\omega_{1}))(\omega_{2})\coloneqq z_{0}$
    when $N(\omega_{1})=\infty$, is a member of $\mathcal{A}(t)$
    for $\mathbb{P}_{\bar{t},t}$-almost all $\omega_{1}$.
\end{itemize}
Before we give the next definition, we mention that for
$\tau,\tau^{\prime},\tau^{\prime\prime}\in\mathscr{T}_{t}([t,\infty])$,
we write $\tau^{\prime}\equiv\tau^{\prime\prime}$ on $\llbracket
t,\tau\rrbracket$
(resp., $\llbracket t,\tau\llbracket$)
if the claim
\[
  \mathbf{1}_{\{\tau^{\prime}\leq
  s\}}(\omega)=\mathbf{1}_{\{\tau^{\prime\prime}\leq
  s\}}(\omega)\text{ for all }t\leq s\leq\tau(\omega)
  \text{ (resp., }<\tau(\omega)\text{)}
\]
holds for $\mathbb{P}$-almost all $\omega$.
\begin{defn}[Control identification]
  For $b,b^{\prime}\in\mathcal{B}(t)$ and $\tau\in\mathscr{T}_{t}([t,T])$,
  we write $b\equiv b^{\prime}$ on $\llbracket t,\tau\rrbracket$ if
  the claim
  \[
    b_{s}(\omega)=b_{s}^{\prime}(\omega)\text{ for almost every
    }s\in[t,\tau(\omega)]
  \]
  holds for $\mathbb{P}$-almost all $\omega$.
  Since inclusion of the endpoint is immaterial, we also write
  $b\equiv b^{\prime}$ on $\llbracket t,\tau\llbracket$ to mean the same thing.

  Similarly, for $a\coloneqq(\tau_{j},z_{j})_{j}$ and
  $a^{\prime}\coloneqq(\tau_{j}^{\prime},z_{j}^{\prime})_{j}$
  in $\mathcal{A}(t)$ and $\tau\in\mathscr{T}_{t}([t,T])$, we write
  $a\equiv a^{\prime}$ on $\llbracket t,\tau\rrbracket$
  (resp., $\llbracket t,\tau\llbracket$) if
  $\tau_{j}\equiv\tau_{j}^{\prime}$
  on $\llbracket t,\tau\rrbracket$
  (resp., $\llbracket t,\tau\llbracket$)
  for each $j$ and the claim
  \[
    z_{j}(\omega)=z_{j}^{\prime}(\omega)\text{ for each }j\text{ such
    that }\tau_{j}(\omega)\leq\tau(\omega)
    \text{ (resp., }<\tau(\omega)\text{)}
  \]
  holds for $\mathbb{P}$-almost all $\omega$.
\end{defn}
\begin{defn}[Strategies]
  $\alpha:\mathcal{B}(t)\rightarrow\mathcal{A}(t)$ is an impulse
  strategy if it is
  \begin{enumerate}[label=(\roman{enumi}),ref=(\roman{enumi}),start=1]

    \item strongly nonanticipative: for each $b,b^{\prime}\in\mathcal{B}(t)$
      and $\tau\in\mathscr{T}_{t}([t,T])$, $\alpha(b)\equiv\alpha(b^{\prime})$
      on $\llbracket t,\tau\rrbracket$ whenever $b\equiv b^{\prime}$ on
      $\llbracket t,\tau\rrbracket$;

    \item delayed: if $t<T$, there is a partition $t=t_{0}<t_{1}<\cdots<t_{m}=T$
      such that for each $b,b^{\prime}\in\mathcal{B}(t)$ and $i<m$,
      $\alpha(b)\equiv\alpha(b^{\prime})$
      on $\llbracket t,t_{i+1}\llbracket$ whenever $b\equiv
      b^{\prime}$ on $\llbracket t,t_{i}\rrbracket$;

    \item an r-strategy: for each $\bar{t}\in[0,t)$ and stochastic
      control $b\in\mathcal{B}(\bar{t})$, the family
      $\alpha(b|_{[t,T]})$ defines a member of
      $\mathcal{A}(\bar{t})$. That is, if
      $\alpha(b|_{[t,T]}(\omega_{1}))=(\tau_{j}(\omega_{1}),z_{j}(\omega_{1}))_{j\geq1}$,
      then $\alpha(b|_{[t,T]})$ denotes the impulse control
      $(\hat{\tau}_{j},\hat{z}_{j})_{j\geq1}\in\mathcal{A}(\bar{t})$
      given by
      $\hat{\tau}_{j}(\omega_{1},\omega_{2})\coloneqq\tau_{j}(\omega_{1})(\omega_{2})$
      and
      $\hat{z}_{j}(\omega_{1},\omega_{2})\coloneqq
      z_{j}(\omega_{1})(\omega_{2})$.
  \end{enumerate}
  The set of all such strategies is denoted $\mathscr{A}(t)$.

  Moreover, for each integer $q\geq1$ and Borel set $Q\subset Z$,
  it is useful to define $\mathscr{A}(t,x)$, $\mathscr{A}^{q}(t,x)$,
  and $\mathscr{A}^{q,Q}(t)$ as the set of all impulse strategies
  $\alpha\in\mathscr{A}(t)$
  with $\operatorname{range}(\alpha)\subset\mathcal{A}(t,x)$,
  $\operatorname{range}(\alpha)\subset\mathcal{A}^{q}(t,x)$,
  and $\operatorname{range}(\alpha)\subset\mathcal{A}^{q,Q}(t)$, respectively.

  $\beta:\mathcal{A}(t)\rightarrow\mathcal{B}(t)$ is a stochastic strategy
  if it is
  \begin{enumerate}[label=(\roman{enumi}),ref=(\roman{enumi}),start=1]

    \item strongly nonanticipative;

    \item delayed;

    \item an r-strategy: for each $\bar{t}\in[0,t)$ and impulse control
      $a\in\mathcal{A}(\bar{t})$, the process $\beta(a|_{(t,T]})$
      is $(\mathscr{F}_{\bar{t},s})_{s\in[\bar{t},T]}$-progressively measurable.
      Here, $\beta(a|_{(t,T]})$ denotes the process
      $(\omega_{1},\omega_{2},s)\mapsto\beta(a|_{(t,T]}(\omega_{1}))(\omega_{2})_{s}$.
  \end{enumerate} The set of all such strategies is denoted $\mathscr{B}(t)$.
\end{defn}
\emph{Strong nonanticipativity} disallows a player from using future
information from their opponent's control. Further explanation is
given in \citep{MR2086176}. Strategies with \emph{delay} are used
to ensure that the upper value of the game is no less than the lower
value.\emph{ r-strategies} (for restricted) were introduced in
\citep[Definition 1.7]{MR997385}
to overcome certain measurability issues.

We introduce below the concept of a \emph{strongly nonanticipative
family of stopping times} that formalizes the intuitive notion that
the decision to stop should not depend on future information from
the controls.  These are used to ensure that the strategies constructed
in the proof of the DPP are strongly nonanticipative. The idea is
similar to the definition of admissible stopping strategies in
\citep[Section 3]{MR3036989},
which serve a similar purpose.
\begin{defn}
  \label{def:strongly_nonanticipative_family_of_stopping_times}Let
  $t\in[0,T]$ and
  $\mathcal{S}\coloneqq\{\theta^{a,b}\}_{(a,b)\in\mathcal{A}(t)\times\mathcal{B}(t)}$
  be a subset of $\mathscr{T}_{t}([t,T])$ whose members are indexed
  by the control tuple $(a,b)\in\mathcal{A}(t)\times\mathcal{B}(t)$.
  We say $\mathcal{S}$ is a strongly nonanticipative family of stopping
  times if for each
  $(a,b),(a^{\prime},b^{\prime})\in\mathcal{A}(t)\times\mathcal{B}(t)$
  and $\tau\in\mathscr{T}_{t}([t,T])$,
  \[
    \theta^{a,b}\equiv\theta^{a^{\prime},b^{\prime}}\text{ on
    }\left\llbracket t,\tau\right\rrbracket \text{ whenever }a\equiv
    a^{\prime}\text{ and }b\equiv b^{\prime}\text{ on
    }\left\llbracket t,\tau\right\rrbracket .
  \]
  Moreover, $\mathcal{S}$ is stable under restriction: for each
  $\bar{t}\in[0,t)$, $a\in\mathcal{A}(\bar{t})$, and
  $b\in\mathcal{B}(\bar{t})$, the family
  $\theta^{a|_{(t,T]},b|_{[t,T]}}$ defines a member of
  $\mathscr{T}_{\bar{t}}([t,T])$. Here,
  $\theta^{a|_{(t,T]},b|_{[t,T]}}$ denotes the stopping time
  $\hat{\theta}$ given by
  $\hat{\theta}(\omega_{1},\omega_{2})\coloneqq
  \theta^{a|_{(t,T]}(\omega_{1}),b|_{[t,T]}(\omega_{1})}(\omega_{2})$.
\end{defn}
We are now ready to introduce the upper and lower values of the game:
\[
  v^{+}(t,x)\coloneqq\sup_{q\geq1}\adjustlimits\inf_{\beta\in\mathscr{B}(t)}\sup_{a\in\mathcal{A}^{q}(t,x)}J(t,x;a,\beta(a))\text{
    and
  }v^{-}(t,x)\coloneqq\sup_{q\geq1}\adjustlimits\sup_{\alpha\in\mathscr{A}^{q}(t,x)}\inf_{b\in\mathcal{B}(t)}J(t,x;\alpha(b),b).
\]
The game is said to admit a value if $v^{+}=v^{-}$ pointwise (on
$[0,T]\times\mathbb{R}^{d}$).

We gather some assumptions below, which are understood to hold throughout
the text.

\begin{assumption}\label{assu:stochastic_assumptions}
  \begin{enumerate}[label=(\roman{enumi}),ref=(\roman{enumi}),start=1]

    \item $B\subset\mathbb{R}^{d_{B}}$ is compact and nonempty;

    \item $\mu:\mathbb{R}^{d}\times B\rightarrow\mathbb{R}^{d}$ and
      $\sigma:\mathbb{R}^{d}\times B\rightarrow\mathbb{R}^{d\times d_{W}}$
      are Lipschitz in $x$ (uniformly in $b$):
      \[
        \left|\mu(x,b)-\mu(y,b)\right|+\left|\sigma(x,b)-\sigma(y,b)\right|\leq\operatorname{const.}\left|x-y\right|
      \]
      and continuous;

    \item $f:[0,T]\times\mathbb{R}^{d}\times B\rightarrow\mathbb{R}$
      and $g:\mathbb{R}^{d}\rightarrow\mathbb{R}$ are bounded and continuous.

  \end{enumerate}
\end{assumption}

\begin{assumption}\label{assu:impulse_assumptions}
  \begin{enumerate}[label=(\roman{enumi}),ref=(\roman{enumi}),start=1]

    \item $Z\subset\mathbb{R}^{d_{Z}}$ is closed and nonempty;

    \item $\Gamma:[0,T]\times Z\rightarrow\mathbb{R}^{d}$ and
      $K:[0,T]\times Z\rightarrow\mathbb{R}$
      are continuous;

    \item \label{enu:negative_growth_condition} $z\mapsto K(t,z)\in\omega(1)$
      as $|z|\rightarrow\infty$ (uniformly in $t$)\footnote{Here,
        $\omega$ is the Bachmann\textendash Landau symbol. Precisely,
        we mean that for each $c>0$, there exists an $r>0$ such that for
      all $(t,z)\in[0,T]\times Z$ with $|z|>r$, $|K(t,z)|\geq c$.}
      and there exists a positive constant $K_{0}$ such that $K\leq-K_{0}$
      pointwise.

  \end{enumerate}
\end{assumption}

\begin{assumption}\label{assu:dpp_and_dpe_assumptions}
  \begin{enumerate}[label=(\roman{enumi}),ref=(\roman{enumi}),start=1]

    \item $f$ and $g$ are Lipschitz in $x$ (uniformly in $t$ and
      $b$):
      \[
        \left|f(t,x,b)-f(t,y,b)\right|+\left|g(x)-g(y)\right|\leq\operatorname{const.}\left|x-y\right|;
      \]

    \item\label{enu:multiple_impulses_are_suboptimal}for each $t\in[0,T]$
      and $z_{1},z_{2}\in Z$, there exists $z\in Z$ such that
      $\Gamma(t,z)=\Gamma(t,z_{1})+\Gamma(t,z_{2})$
      and $K(t,z)\geq K(t,z_{1})+K(t,z_{2})$;

    \item\label{enu:enforce_boundary_conditions}for each $(t_{n},x_{n})_{n}$
      in $[0,T]\times\mathbb{R}^{d}$ converging to $(T,x)$,
      $\liminf_{n\rightarrow\infty}v^{-}(t_{n},x_{n})\geq g(x)$
      and if $v^{+}(t_{n},x_{n})>\mathcal{M}v^{+}(t_{n},x_{n})$ for all
      $n$, $\limsup_{n\rightarrow\infty}v^{+}(t_{n},x_{n})\leq g(x)$.

  \end{enumerate}
\end{assumption}

Before continuing, we discuss briefly the significance of the assumptions
listed above.
\begin{itemize}
  \item \prettyref{assu:stochastic_assumptions} and
    \prettyref{assu:impulse_assumptions}
    are used throughout. In particular, \prettyref{assu:impulse_assumptions}
    \ref{enu:negative_growth_condition} ensures that optimal impulses
    are contained in a compact set.
  \item \prettyref{assu:dpp_and_dpe_assumptions} is used only to establish
    the DPP and DPE; it is not needed in the proof of the comparison principle.
    \ref{enu:multiple_impulses_are_suboptimal} ensures that multiple
    impulses occurring at the same time are suboptimal.
    \ref{enu:enforce_boundary_conditions},
    which is identical to \citep[Assumption (E3)]{MR2568293}, is introduced
    to avoid a detailed analysis of the value functions at the terminal
    time, which is not the main focus of this work.
    \prettyref{exa:terminal_continuity}
    presents a situation in which \ref{enu:enforce_boundary_conditions}
    is satisfied.
\end{itemize}
For a locally bounded above (resp. below) function $u$ from some
metric space $Y$ to $\mathbb{R}$, we use $u^{*}$ (resp. $u_{*}$)
to denote the upper (resp. lower) semicontinuous envelope of $u$.
Unless otherwise mentioned, $Y$ is taken to be $[0,T]\times\mathbb{R}^{d}$.

We are now in a position to state the main results of this work. For
brevity, let $\mathcal{O}\coloneqq[0,T)\times\mathbb{R}^{d}$,
$\partial^{+}\mathcal{O}\coloneqq\{T\}\times\mathbb{R}^{d}$
denote the parabolic boundary of $\mathcal{O}$, and
\begin{equation}
  J_{0}(t,x;a,b;\theta)\coloneqq\int_{t}^{\theta}f(s,X_{s},b_{s})ds+\sum_{\tau_{j}\leq\theta}K(\tau_{j},z_{j})\label{eq:J_0}
\end{equation}
where $\theta$ is a stopping time in $\mathscr{T}_{t}([t,T])$.
We begin with a DPP whose proof is given in \prettyref{sec:dpp_proof}.
\begin{thm}[DPP]
  \label{thm:dpp}There exists a nonempty compact set $Q\subset Z$
  such that for each $(t,x)\in\mathcal{O}$, finite nonempty set
  $R\subset[t,T)$, and each strongly nonanticipative family of stopping
  times
  $\{\theta^{a,b}\}_{(a,b)\in\mathcal{A}(t)\times\mathcal{B}(t)}\subset\mathscr{T}_{t}(R)$,
  the following statements hold:

  \begin{enumerate}[label=(\roman{enumi}),ref=(\roman{enumi}),start=1]
    \item letting $X\coloneqq X^{t,x;a,\beta(a)}$
      and $\theta\coloneqq\theta^{a,\beta(a)}$,
      \[
        \adjustlimits\inf_{\beta\in\mathscr{B}(t)}\sup_{a\in\mathcal{A}^{q,Q}(t)}J(t,x;a,\beta(a))
        \leq\adjustlimits\inf_{\beta\in\mathscr{B}(t)}\sup_{a\in\mathcal{A}^{q,Q}(t)}\mathbb{E}\left[J_{0}(t,x;a,\beta(a);\theta)+v^{+}(\theta,X_{\theta})\right]
      \]
      for each integer $q\geq1$ and hence
      \[
        v^{+}(t,x)\leq\sup_{q\geq1}\adjustlimits\inf_{\beta\in\mathscr{B}(t)}\sup_{a\in\mathcal{A}^{q,Q}(t)}\mathbb{E}\left[J_{0}(t,x;a,\beta(a);\theta)+v^{+}(\theta,X_{\theta})\right];
      \]

    \item letting $X\coloneqq X^{t,x;\alpha(b),b}$ and
      $\theta\coloneqq\theta^{\alpha(b),b}$,
      \[
        v^{-}(t,x)\geq\sup_{q\geq1}\adjustlimits\sup_{\alpha\in\mathscr{A}^{q,Q}(t)}\inf_{b\in\mathcal{B}(t)}\mathbb{E}\left[J_{0}(t,x;\alpha(b),b;\theta)+v^{-}(\theta,X_{\theta})\right].
      \]
  \end{enumerate}
\end{thm}
The HJBI-QVI associated with the game is a quasi-variational inequality:
\begin{equation}
  0=F(\cdot,u,Du(\cdot),D^{2}u(\cdot))\coloneqq
  \begin{cases}
    \min\{-\inf_{b\in
    B}\{(\partial_{t}+\text{\L}^{b})u+f^{b}\},u-\mathcal{M}u\} &
    \text{on }\mathcal{O}\\
    \min\{u-g,u-\mathcal{M}u\} & \text{on }\partial^{+}\mathcal{O}
  \end{cases}\label{eq:formal_hjbi}
\end{equation}
where $f^{b}(t,x)\coloneqq f(t,x,b)$, $\mathcal{M}$ is defined by
\eqref{eq:intervention}, and
\[
  \text{\L}^{b}u(t,x)\coloneqq\frac{1}{2}\operatorname{trace}(\sigma(x,b)\sigma^{\intercal}(x,b)D_{x}^{2}u(t,x))+\left\langle
  \mu(x,b),D_{x}u(t,x)\right\rangle .
\]
We point out that due to the operator $\mathcal{M}$, \eqref{eq:formal_hjbi}
is nonlocal in its use of $u$. No second time derivatives appear
and so we use $D^{2}$ and $D_{x}^{2}$ interchangeably, while $D$
is interpreted to mean either $(\partial_{t},D_{x})$ or $D_{x}$,
depending on context. Since the above is only formal, we need to ascribe
meaning to the notion of a ``solution'' to \eqref{eq:formal_hjbi}:
\begin{defn}[Viscosity solution]
  \label{def:viscosity_solution} A locally bounded above (resp. below)
  function $u:\operatorname{cl}\mathcal{O}\rightarrow\mathbb{R}$ is
  a viscosity subsolution\emph{ }(resp. supersolution) of \eqref{eq:formal_hjbi}
  if, letting $w\coloneqq u^{*}$ (resp. $w\coloneqq u_{*}$), for all
  $(t,x,\varphi)\in\mathcal{O}\times C^{1,2}(\mathcal{O})$ such that
  $(w-\varphi)(t,x)$ is a local maximum (resp. minimum) of
  $w-\varphi$,\footnote{The symbol $w-\varphi$ should be understood
    subject to the convention
    $f_{A}+f_{B}\coloneqq f_{A}|_{B}+f_{B}$ whenever
    $f_{A}:A\rightarrow\mathbb{R}$,
  $f_{B}:B\rightarrow\mathbb{R}$, and $B\subset A$.}
  \[
    F(t,x,w,D\varphi(t,x),D^{2}\varphi(t,x))\leq0\text{ (resp. }\geq0\text{)}
  \]
  and for all $(t,x)\in\partial^{+}\mathcal{O}$, $F(t,x,w)\leq0$ (resp.
  $\geq0$).\footnote{We have abused notation slightly in writing
    $F(t,x,w)$ since the
    derivatives $(Du,D^{2}u)$ do not appear on the boundary
  $\partial^{+}\mathcal{O}$.}

  $u$ is said to be a viscosity solution of \eqref{eq:formal_hjbi}
  whenever it is simultaneously a supersolution and subsolution of
  \eqref{eq:formal_hjbi}.
\end{defn}
We hereafter drop the term ``viscosity'' in our discussions, since
it is the main solution concept in this work.

We can now state the relationship between the upper and lower value
functions of the game and the HJBI-QVI; namely that the HJBI-QVI is
the DPE associated with the game. A proof of this fact is given in
\prettyref{sec:dpe_proof}.
\begin{thm}[DPE]
  \label{thm:dpe}$v^{+}$ (resp. $v^{-}$) is a bounded subsolution
  (resp. supersolution) of \eqref{eq:formal_hjbi}.
\end{thm}
We also establish a comparison principle for the HJBI-QVI, with proof
given in \prettyref{sec:comparison_principle_proof}.
\begin{thm}[Comparison principle]
  \label{thm:comparison_principle}If $u$ is a bounded subsolution
  and $w$ is a bounded supersolution of \eqref{eq:formal_hjbi},
  $u^{*}\leq w_{*}$
  pointwise.
\end{thm}
We can now establish that the game admits a value.
\begin{thm}[Value of the game]
  \label{thm:value} $v^{+}=v^{-}$ pointwise.
\end{thm}
\begin{proof}
  Due to the delay in the strategies, it is easy to establish that for
  each $t\in[0,T]$ and $(\alpha,\beta)\in\mathscr{A}(t)\times\mathscr{B}(t)$,
  there exists a control pair
  $(a,b)\in\mathcal{A}(t)\times\mathcal{B}(t)$
  such that, on $\llbracket t, T\rrbracket$, $\alpha(b)\equiv a$ and
  $\beta(a)\equiv b$.

  To construct the control pair, assume $t < T$ (the case of $t = T$
  is immediate), pick an arbitrary $b^0 \in \mathcal{B}(t)$, and define
  \[
    a^n \coloneqq \alpha(b^{n - 1}) \text{ and } b^n \coloneqq \beta(a^n)
    \text{ for } n \geq 1.
  \]
  Denote by $t = t_0 < \cdots < t_m = T$ a common refinement of the
  delay partitions for $\alpha$ and $\beta$.
  By induction as in the proof of \citep[Lemma 2.4]{MR2086176}, one
  obtains $a^{i + 1} \equiv a^i$ and $b^{i + 1} \equiv b^i$ on
  $\llbracket t,t_i\llbracket$ for $1 \leq i \leq m$.
  In particular, $b^{m + 1} \equiv b^m$ on $\llbracket t,T\llbracket$.
  Since this is the same as $b^{m + 1} \equiv b^m$ on $\llbracket
  t,T\rrbracket$, by strong nonanticipativity, $a^{m + 2} \equiv a^{m
  + 1}$ on $\llbracket t,T\rrbracket$.
  In other words, $\alpha(b^{m + 1}) \equiv a^{m + 1}$ and
  $\beta(a^{m + 1}) \equiv b^{m + 1}$ on $\llbracket t,T\rrbracket$, as desired.

  Now, let $(t,x)\in\operatorname{cl}\mathcal{O}$ and $\epsilon>0$.
  Choose an integer $q\geq1$ and
  $(\alpha,\beta)\in\mathscr{A}^{q}(t,x)\times\mathscr{B}(t)$
  such that
  \[
    v^{-}(t,x)\leq\inf_{b\in\mathcal{B}(t)}J(t,x;\alpha(b),b)+\epsilon/2\text{
      and
    }v^{+}(t,x)\geq\sup_{a\in\mathcal{A}^{q}(t,x)}J(t,x;a,\beta(a))-\epsilon/2.
  \]
  By the remark at the beginning of this proof, we can find
  $(a_{0},b_{0})\in\mathcal{A}^{q}(t,x)\times\mathcal{B}(t)$
  such that $\alpha(b_{0})\equiv a_{0}$ and $\beta(a_{0})\equiv
  b_{0}$ on $\llbracket t, T\rrbracket$. Therefore,
  \[
    v^{-}(t,x)\leq
    J(t,x;\alpha(b_{0}),b_{0})+\epsilon/2=J(t,x;a_{0},b_{0})+\epsilon/2=J(t,x;a_{0},\beta(a_{0}))+\epsilon/2\leq
    v^{+}(t,x)+\epsilon.
  \]
  Since $(t,x)$ and $\epsilon$ were arbitrary, we conclude that
  $v^{-}\leq v^{+}$
  pointwise. The reverse inequality is an immediate consequence of Theorems
  \ref{thm:dpe} and \ref{thm:comparison_principle}.
\end{proof}
A consequence of the above results is the following existence and
uniqueness claim:
\begin{cor}[HJBI-QVI existence and uniqueness]
  $v^{+}=v^{-}$ is a continuous solution of \eqref{eq:formal_hjbi},
  unique among all bounded solutions.
\end{cor}

\section{\label{sec:regularity}Regularity}

We first make clear the default norms used in this work:

\begin{notation*}Let $\langle x,y\rangle$ be the Euclidean inner
  product, $|x|\coloneqq\sqrt{\langle x,x\rangle}$, and $B(x;r)$ be
  the ball (in the associated metric) of radius $r>0$ centred at $x$.
\end{notation*}

Before we begin, we make the observation that any of the Lipschitz
functions defined in \prettyref{sec:framework} are necessarily of
linear growth. For example,
\[
  \left|\mu(x,b)\right|\leq\left|\mu(x,b)-\mu(0,b)\right|+\left|\mu(0,b)\right|\leq\operatorname{const.}\left|x\right|+\sup_{b\in
  B}\left|\mu(0,b)\right|,
\]
so that the claim follows from the continuity of $\mu$ and compactness
of $B$.

\begin{lem}
  \label{lem:bounded_value}$v^{+}$ and $v^{-}$ are bounded.
\end{lem}
\begin{proof}
  Let $c\coloneqq T\Vert f\Vert_{\infty}+\Vert g\Vert_{\infty}$ and
  $(t,x)\in\operatorname{cl}\mathcal{O}$. Since the sup-player is free
  to play a control $\hat{a}\in\mathcal{A}(t,x)$ with no impulses (i.e.,
  $(\#\hat{a})_{T}=0$),
  \begin{equation}
    v^{+}(t,x)\geq\inf_{\beta\in\mathscr{B}(t)}\mathbb{E}\left[\int_{t}^{T}f(s,X_{s},\beta(\hat{a})_{s})ds+g(X_{T})\right]\geq-c\label{eq:lower_bound}
  \end{equation}
  where it is understood that $X\coloneqq X^{t,x;\hat{a},\beta(\hat{a})}$.
  Moreover, since $K\leq0$,
  \begin{equation}
    v^{+}(t,x)\leq\sup_{q\geq1}\adjustlimits\inf_{\beta\in\mathscr{B}(t)}\sup_{a\in\mathcal{A}^{q}(t,x)}\mathbb{E}\left[\int_{t}^{T}f(s,X_{s},\beta(a)_{s})ds+g(X_{T})\right]\leq
    c\label{eq:upper_bound}
  \end{equation}
  where it is understood that $X\coloneqq X^{t,x;a,\beta(a)}$. The
  same arguments hold for the lower value function $v^{-}$.
\end{proof}

We defined the notion of an admissible control based on the existence
and uniqueness of a strong solution to the SDE. It is not possible
for us to perform our analyses on these control sets without running
into trouble. The remaining results of this section alleviate this
by establishing some results related to $\mathcal{A}^{q,Q}(t)$.
\begin{lem}
  \label{lem:existence_and_uniqueness_of_sde}Let $q\geq1$ be an integer,
  $Q\subset Z$ be a nonempty compact set, $(t,x)\in[0,T]\times\mathbb{R}^{d}$,
  and $(a,b)\in\mathcal{A}^{q,Q}(t)\times\mathcal{B}(t)$. Then, there
  exists a unique solution to \eqref{eq:sde}.
\end{lem}
\begin{proof}
  Let $X$ and $\hat{X}$ be two solutions (with the same initial condition
  $(t,x)$ and controls $(a,b)$). Letting $\delta X_{s}\coloneqq
  X_{s}-\hat{X}_{s}$,
  note that
  \[
    X_{s}-\hat{X}_{s}=\int_{t}^{s}\mu(X_{u},b_{u})-\mu(\hat{X}_{u},b_{u})du+\int_{t}^{s}\sigma(X_{u},b_{u})-\sigma(\hat{X}_{u},b_{u})dW_{u}
  \]
  we can mimic the argument in the proof of \citep[Theorem 5.2.1]{MR2001996}
  to arrive at
  \[
    \mathbb{P}(\delta X_{s}=0\text{ for all }s\in\mathbb{Q}\cap[t,T])=1.
  \]
  Since the paths $X$ and $\hat{X}$ were presumed to be right continuous,
  so too are the paths of $\delta X$. Therefore,
  \[
    \mathbb{P}(\delta X_{s}=0\text{ for all }s\in[t,T])=1
  \]
  and hence uniqueness is proved.

  Let $Y_{s}^{(0)}\coloneqq x+\sum_{\tau_{j}\leq s}\Gamma(\tau_{j},z_{j})$
  and define inductively
  \[
    Y_{s}^{(k+1)}\coloneqq
    x+\int_{t}^{s}\mu(Y_{u}^{(k)},b_{u})du+\int_{t}^{s}\sigma(Y_{u}^{(k)},b_{u})dW_{u}+\sum_{\tau_{j}\leq
    s}\Gamma(\tau_{j},z_{j}).
  \]
  Before we continue, we should check that
  $Y^{(k+1)}\in\mathbb{L}^{2}(\Omega_{t,T}\times[t,T])$
  whenever $Y^{(k)}\in\mathbb{L}^{2}(\Omega_{t,T}\times[t,T])$. Using
  the linear growth of $\mu$ and $\sigma$, we can show that
  \[
    \mathbb{E}\left[\int_{t}^{T}\left|\int_{t}^{s}\mu(Y_{u}^{(k)},b_{u})du\right|^{2}ds\right]+\mathbb{E}\left[\int_{t}^{T}\left|\int_{t}^{s}\sigma(Y_{u}^{(k)},b_{u})dW_{u}\right|^{2}ds\right]<\infty.
  \]
\begin{comment}
\[
\mathbb{E}\left[\int_{t}^{T}\left|\int_{t}^{s}\mu(Y_{u}^{(k)},b_{u})du\right|^{2}ds\right]\leq\operatorname{const.}\mathbb{E}\left[\int_{t}^{T}\int_{t}^{s}\left|\mu(Y_{u}^{(k)},b_{u})\right|^{2}duds\right]\leq\operatorname{const.}\mathbb{E}\left[\int_{t}^{T}\left(1+\left|Y_{u}^{(k)}\right|\right)^{2}du\right]
\]
\end{comment}
  Moreover,
  \[
    \mathbb{E}\left[\int_{t}^{T}\left|{\textstyle \sum_{\tau_{j}\leq
    s}}\Gamma(\tau_{j},z_{j})\right|^{2}ds\right]\leq\operatorname{const.}\sup_{(t,z)\in[0,T]\times
    Q}\left|\Gamma(t,z)\right|^{2}<\infty
  \]
  where $\operatorname{const.}$ can depend on $q$. The above also
  shows $Y^{(0)}\in\mathbb{L}^{2}(\Omega_{t,T}\times[t,T])$.

  The remainder of the proof is identical to \citep[Theorem 5.2.1]{MR2001996}
  with the exception that Doob's inequality for càdlàg martingales
  \citep[Theorem II.1.7]{MR1303781}
  should be used instead of \citep[Theorem 3.2.4]{MR2001996}.
\end{proof}
In the sequel, we need to construct a countable Borelian partition
on which players can make ``$\epsilon$-optimal'' choices. To do
so, we require $J$ to exhibit uniform continuity in space, independent
of the other arguments:
\begin{lem}
  \label{lem:lipschitz_functional}$|J(t,x;a,b)-J(t,\hat{x};a,b)|\leq\operatorname{const.}|x-\hat{x}|$
  for all $t\in[0,T]$, integers $q\geq1$, nonempty compact sets $Q\subset Z$,
  and $(a,b)\in\mathcal{A}^{q,Q}(t)\times\mathcal{B}(t)$.
\end{lem}
\begin{proof}
  Let $X\coloneqq X^{t,x;a,b}$ and $\hat{X}\coloneqq X^{t,\hat{x};a,b}$.
  Further letting $\delta x\coloneqq x-\hat{x}$, $\delta
  X_{s}\coloneqq X_{s}-\hat{X}_{s}$,
  $\delta\mu_{s}\coloneqq\mu(X_{s},b_{s})-\mu(\hat{X}_{s},b_{s})$,
  and $\delta\sigma_{s}\coloneqq\sigma(X_{s},b_{s})-\sigma(\hat{X}_{s},b_{s})$,
  we can use Hölder's inequality, Itô isometry, the Lipschitz continuity
  of $\mu$ and $\sigma$, and the Fubini-Tonelli theorem to get
  \begin{align*}
    \mathbb{E}\left[\left|\delta X_{s}\right|^{2}\right] &
    \leq\operatorname{const.}\mathbb{E}\left[\left|\delta
    x\right|^{2}+\left|\int_{t}^{s}\delta\mu_{u}du\right|^{2}+\left|\int_{t}^{s}\delta\sigma_{u}dW_{u}\right|^{2}\right]\\
    & \leq\operatorname{const.}\mathbb{E}\left[\left|\delta
    x\right|^{2}+T\int_{t}^{s}\left|\delta\mu_{u}\right|^{2}du+\int_{t}^{s}\left|\delta\sigma_{u}\right|^{2}du\right]\\
    & \leq\operatorname{const.}\left(\left|\delta
      x\right|^{2}+\int_{t}^{s}\mathbb{E}\left[\left|\delta
    X_{u}\right|^{2}\right]du\right)
  \end{align*}
  for $s\in[t,T]$. Now, an application of Grönwall's lemma followed
  by Jensen's inequality (for concave functions) yields%
\begin{comment}
\begin{align*}
\mathbb{E}\left[\left|\delta X_{s}\right|^{2}\right] & \leq\operatorname{const.}\left|\delta x\right|^{2} & \text{(Grönwall)}\\
\mathbb{E}\left[\left|\delta X_{s}\right|\right]\leq\sqrt{\mathbb{E}\left[\left|\delta X_{s}\right|^{2}\right]} & \leq\operatorname{const.}\left|\delta x\right| & \text{(Jensen)}
\end{align*}
\end{comment}
  \begin{equation}
    \mathbb{E}\left[\left|\delta
    X_{s}\right|\right]\leq\operatorname{const.}\left|\delta
    x\right|\text{ for }s\in[t,T].\label{eq:lipschitz_sde}
  \end{equation}
  Note, in particular, that $\operatorname{const.}$ depends only on
  quantities such as $T$ and the Lipschitz constants of $\mu$ and
  $\sigma$. Moreover, defining $\delta f_{s}\coloneqq
  f(s,X_{s},b_{s})-f(s,\hat{X}_{s},b_{s})$
  and $\delta g_{T}\coloneqq g(X_{T})-g(\hat{X}_{T})$,
  \begin{align*}
    \left|J(t,x;a,b)-J(t,\hat{x};a,b)\right| &
    \leq\mathbb{E}\left[\int_{t}^{T}\left|\delta
    f_{s}\right|ds+\left|\delta g_{T}\right|\right]
  \end{align*}
\begin{comment}
$\leq\operatorname{const.}\left(\int_{t}^{T}\mathbb{E}\left[\left|\delta X_{s}\right|\right]ds+\mathbb{E}\left[\left|\delta X_{T}\right|\right]\right)$
\end{comment}
  and the desired result follows from applying the Lipschitz continuity
  of $f$ and $g$ and \eqref{eq:lipschitz_sde} to the inequality above.
\end{proof}
\begin{rem}
  \label{rem:a_priori_continuity}While the above implies the Lipschitz
  continuity of the value functions in $x$ for each fixed $t$, it
  does not imply uniform continuity on $[0,T]\times\mathbb{R}^{d}$!
  Such continuity requires additional conditions on $K$ (see, e.g.,
  \citep[(2.6)]{MR3053571}), which we avoid.
\end{rem}
\begin{lem}
  \label{lem:bounded_functional}Let $q\geq1$ be an integer, $Q\subset Z$
  be a nonempty compact set, $c\coloneqq T\Vert f\Vert_{\infty}+\Vert
  g\Vert_{\infty}$,
  and $K_{1}\coloneqq\sup_{[0,T]\times Q}|K(t,z)|$. Then,
  $|J(t,x;a,b)|\leq c+qK_{1}$
  for all $(t,x)\in\operatorname{cl}\mathcal{O}$ and
  $(a,b)\in\mathcal{A}^{q,Q}(t)\times\mathcal{B}(t)$.
\end{lem}
\begin{proof}
  Let $(t,x)\in\operatorname{cl}\mathcal{O}$ and
  $(a,b)\in\mathcal{A}^{q,Q}(t)\times\mathcal{B}(t)$.
  Since $K$ is nonpositive, $J(t,x;a,b)\leq c$. Moreover,
  $J(t,x;a,b)\geq-c-\mathbb{E}\left[(\#a)_{T}\right]K_{1}\geq-c-qK_{1}$.
\end{proof}
\begin{lem}
  \label{lem:control_reduction}There exists a nonempty compact set $Q\subset Z$
  such that for all $(t,x)\in\operatorname{cl}\mathcal{O}$,
  \begin{align*}
    v^{+}(t,x) &
    =\sup_{q\geq1}\adjustlimits\inf_{\beta\in\mathscr{B}(t)}\sup_{a\in\mathcal{A}^{q,Q}(t)}J(t,x;a,\beta(a))\\
    \text{ and }v^{-}(t,x) &
    =\sup_{q\geq1}\adjustlimits\sup_{\alpha\in\mathscr{A}^{q,Q}(t)}\inf_{b\in\mathcal{B}(t)}J(t,x;\alpha(b),b).
  \end{align*}
\end{lem}
While its proof is rather technical, this result is intuitive: since
$K$ is guaranteed to grow large and negative
(\prettyref{assu:impulse_assumptions}
\ref{enu:negative_growth_condition}), we should be able to find a
compact set $Q$ such that impulses taking values outside of $Q$
are suboptimal.
\begin{proof}
  Let $c\coloneqq T\Vert f\Vert_{\infty}+\Vert g\Vert_{\infty}$. By
  \prettyref{assu:impulse_assumptions} \ref{enu:negative_growth_condition},
  there is an $r>0$ such that $Z\cap\operatorname{cl}B(0;r)$ is nonempty
  and $K(t,z)\leq-2c$ for all $(t,z)\in[0,T]\times Z$ with $|z|>r$.
  Let $Q\coloneqq Z\cap\operatorname{cl}B(0;r)$.\emph{}

  Let $\epsilon>0$. Choose an integer $q_{\epsilon}\geq1$ and
  $\alpha_{\epsilon}\in\mathscr{A}^{q_{\epsilon}}(t,x)$
  such that
  \begin{equation}
    v^{-}(t,x)\leq\inf_{b\in\mathcal{B}(t)}J(t,x;\alpha_{\epsilon}(b),b)+\epsilon.\label{eq:control_reduction_proof_1}
  \end{equation}
  Let $\alpha_{\epsilon}(b)\coloneqq(\tau_{j}(b),z_{j}(b))_{j}$ denote
  the impulse control that is obtained from applying $b\in\mathcal{B}(t)$
  to the strategy $\alpha_{\epsilon}$. Further letting
  \begin{equation}
    \tau_{j}^{\prime}(b)\coloneqq\tau_{j}(b)\mathbf{1}_{\cap_{j^{\prime}\leq
      j}\{z_{j^{\prime}}(b)\in
    Q\}}+\left(\infty\right)\mathbf{1}_{\cup_{j^{\prime}\leq
    j}\{z_{j^{\prime}}(b)\notin Q\}},\label{eq:control_reduction_proof_2}
  \end{equation}
  we define a new strategy $\alpha$ that maps each $b\in\mathcal{B}(t)$
  to $\alpha(b)\coloneqq(\tau_{j}^{\prime}(b),z_{j}(b))_{j}$. Since
  $\alpha$ is obtained from $\alpha_{\epsilon}$ by truncating after
  the first impulse whose value is outside $Q$, the strong
  nonanticipativity, delay, and r-strategy properties are inherited
  from $\alpha_{\epsilon}$. Note, in particular, that
  $\alpha\in\mathscr{A}^{q_{\epsilon},Q}(t)$. Now,
  define
  \[
    j_{b}\coloneqq\inf\left\{
    j\geq1\colon\tau_{j}^{\prime}(b)\neq\tau_{j}(b)\right\}
  \]
  (note that $j_{b}$ is a random variable). As usual, we use the convention
  that $\inf\emptyset=\infty$. For the rest of this proof, write
  $G(t,x;a,b)$ for the random variable inside the expectation defining
  $J(t,x;a,b)$ so that $J(t,x;a,b)=\mathbb{E}[G(t,x;a,b)]$. Trivially,
  \begin{equation}
    \mathbb{E}\left[\left(G(t,x;\alpha_{\epsilon}(b),b)-G(t,x;\alpha(b),b)\right)\mathbf{1}_{\{j_{b}=\infty\}}\right]=0\text{
    for all }b\in\mathcal{B}(t).\label{eq:control_reduction_proof_new_1}
  \end{equation}
  Moreover, for all $b\in\mathcal{B}(t)$,
  \begin{multline*}
    \mathbb{E}\left[\left(G(t,x;\alpha_{\epsilon}(b),b)-G(t,x;\alpha(b),b)\right)\mathbf{1}_{\{j_{b}<\infty\}}\right]\\
    =\mathbb{E}\Biggl[\Biggl(\int_{\tau_{j_{b}}(b)}^{T}f(s,X_{s}^{t,x;\alpha_{\epsilon}(b),b},b_{s})-f(s,X_{s}^{t,x;\alpha(b),b},b_{s})ds+\sum_{\{j\geq
        j_{b}\colon\tau_{j}(b)\leq T\}}K(\tau_{j}(b),z_{j}(b))\\
    +g(X_{T}^{t,x;\alpha_{\epsilon}(b),b})-g(X_{T}^{t,x;\alpha(b),b})\Biggr)\mathbf{1}_{\{j_{b}<\infty\}}\Biggr]\leq\mathbb{E}\left[\left(2c+\sum_{\{j\geq
            j_{b}\colon\tau_{j}(b)\leq
    T\}}K(\tau_{j}(b),z_{j}(b))\right)\mathbf{1}_{\{j_{b}<\infty\}}\right].
  \end{multline*}
  Using the fact that $K\leq-2c$ on $[0,T]\times(Z\setminus Q)$,
  we have that
  \[
    \left(\sum_{\{j\geq j_{b}\colon\tau_{j}(b)\leq
    T\}}K(\tau_{j}(b),z_{j}(b))\right)\mathbf{1}_{\{j_{b}<\infty\}}\leq
    K(\tau_{j_{b}}(b),z_{j_{b}}(b))\mathbf{1}_{\{j_{b}<\infty\}}\leq-2c\mathbf{1}_{\{j_{b}<\infty\}}.
  \]
  Therefore,
  \begin{equation}
    \mathbb{E}\left[\left(G(t,x;\alpha_{\epsilon}(b),b)-G(t,x;\alpha(b),b)\right)\mathbf{1}_{\{j_{b}<\infty\}}\right]\leq0\text{
    for all }b\in\mathcal{B}(t).\label{eq:control_reduction_proof_new_2}
  \end{equation}
  Combining \eqref{eq:control_reduction_proof_new_1} and
  \eqref{eq:control_reduction_proof_new_2}
  and using $J=\mathbb{E}[G]$, we obtain
  \begin{equation}
    J(t,x;\alpha_{\epsilon}(b),b)\leq J(t,x;\alpha(b),b)\text{ for
    all }b\in\mathcal{B}(t).\label{eq:control_reduction_proof_3}
  \end{equation}
  Now, by \eqref{eq:control_reduction_proof_1} and
  \eqref{eq:control_reduction_proof_3}
  and the arbitrariness of $\epsilon$,
  \[
    v^{-}(t,x)\leq\sup_{q\geq1}\adjustlimits\sup_{\alpha\in\mathscr{A}^{q,Q}(t)}\inf_{b\in\mathcal{B}(t)}J(t,x;\alpha(b),b).
  \]
  The reverse inequality is trivial, since
  $\mathscr{A}^{q,Q}(t)\subset\mathscr{A}^{q}(t,x)$
  by \prettyref{lem:existence_and_uniqueness_of_sde}.

  The equality involving $v^{+}$ is established using similar arguments.
  Namely, it is sufficient to construct, for each $a\in\mathcal{A}^{q}(t,x)$,
  a new control $a^{\prime}\in\mathcal{A}^{q,Q}(t)$ such that
  \begin{equation}
    J(t,x;a,\beta(a))\leq J(t,x;a^{\prime},\beta(a^{\prime}))\text{
    for all }\beta\in\mathscr{B}(t).\label{eq:control_reduction_proof_4}
  \end{equation}
  Letting $a\coloneqq(\tau_{j},z_{j})_{j}$, we take
  $a^{\prime}\coloneqq(\tau_{j}^{\prime},z_{j})_{j}$
  where, similarly to \eqref{eq:control_reduction_proof_2},
  \[
    \tau_{j}^{\prime}\coloneqq\tau_{j}\mathbf{1}_{\cap_{j^{\prime}\leq
      j}\{z_{j^{\prime}}\in
    Q\}}+\left(\infty\right)\mathbf{1}_{\cup_{j^{\prime}\leq
    j}\{z_{j^{\prime}}\notin Q\}}.
  \]
  Define
  \[
    \eta \coloneqq \inf\left\{ \tau_j \colon j \geq 1, \tau_j^\prime
    \neq \tau_j \right\} \wedge T
  \]
  By construction, $a\equiv a^{\prime}$ on $\llbracket t,\eta\llbracket$.
  Since the Brownian filtration is continuous, $\eta$ admits an announcing
  sequence $(\kappa_{n})_{n}$.
  Applying the strong nonanticipativity of $\beta$ on each
  $\llbracket t,\kappa_{n}\rrbracket$
  and then letting $n\rightarrow\infty$, we get
  $\beta(a)\equiv\beta(a^{\prime})$ on $\llbracket t,\eta\llbracket$.
  Using this fact, it is now straightforward to establish
  \eqref{eq:control_reduction_proof_4}.
\end{proof}

\section{\label{sec:dpp_proof}Dynamic programming principle}

We prove, in this section, \prettyref{thm:dpp}. We begin with the
$\leq$ inequality.
\begin{proof}[Proof of \prettyref{thm:dpp} ($\leq$)]
  Let $Q$ be the compact set in the statement of
  \prettyref{lem:control_reduction}
  and let $q$ be a positive integer. Fix $(t,x)\in\mathcal{O}$,
  a finite nonempty set $R\subset[t,T)$, and a strongly nonanticipative
  family of stopping times
  $\{\theta^{a,b}\}_{(a,b)\in\mathcal{A}(t)\times\mathcal{B}(t)}\subset\mathscr{T}_{t}(R)$.
  Let $\epsilon>0$ be arbitrary.

  Let $(r,y)\in R\times\mathbb{R}^{d}$ be arbitrary. The definition
  of $v^{+}$ and
  \prettyref{lem:control_reduction} allow us to choose
  $\beta^{r,y}\in\mathscr{B}(r)$
  such that
  \[
    J(r,y;a,\beta^{r,y}(a))\leq v^{+}(r,y)+\epsilon\text{ for all
    }a\in\mathcal{A}^{q,Q}(r).
  \]
  Since $J$ and $v^{+}$ are Lipschitz continuous in space with Lipschitz
  constant independent of time and controls (see
  \prettyref{lem:lipschitz_functional}),
  there is a positive radius $\rho_{r,y}$ such that
  \begin{equation}
    J(r,y^{\prime};a,\beta^{r,y}(a))\leq
    v^{+}(r,y^{\prime})+2\epsilon\text{ for all }y^{\prime}\in
    B(y;\rho_{r,y})\text{ and }a\in\mathcal{A}^{q,Q}(r).
    \label{eq:dpp_proof_upper_local}
  \end{equation}

  For the remainder of this proof, fix $\beta\in\mathscr{B}(t)$.
  Let $M$ be positive and, for each $a\in\mathcal{A}^{q,Q}(t)$, let
  $X^a\coloneqq X^{t,x;a,\beta(a)}$, $\theta_a\coloneqq
  \theta^{a,\beta(a)}$, and $E_{a}\coloneqq\{|X^a_{\theta_a}-x|>M\}$.
  By Markov's inequality,
  \[
    \sup_{a\in\mathcal{A}^{q,Q}(t)}\mathbb{P}(E_{a})\leq\frac{1}{M^{2}}\sup_{a\in\mathcal{A}^{q,Q}(t)}\mathbb{E}\left[\sup_{s\in[t,T]}\left|X^a_{s}-x\right|^{2}\right].
  \]
  Using the compactness of the control set $B$, the compactness of $Q$,
  and the bound on the number of impulses, standard estimates for the SDE imply
  that the outer supremum on the right hand side of the above inequality
  is finite. Let $c$ and $K_{1}$
  be defined as in \prettyref{lem:bounded_functional} and choose $M$
  sufficiently large so that
  \begin{equation}
    \left(c+qK_{1}+\Vert v^{+}\Vert
    _{\infty}\right)\sup_{a\in\mathcal{A}^{q,Q}(t)}\mathbb{P}(E_{a})\leq\epsilon.
    \label{eq:dpp_proof_upper_tail}
  \end{equation}

  For each $r\in R$, choose finitely many points $y_{r,1},\ldots,y_{r,N_{r}}$
  so that the balls $B(y_{r,j};\rho_{r,y_{r,j}})$ cover
  $\operatorname{cl}B(x;M)$.
  To turn this cover into a Borelian partition, define
  \[
    C_{r,j}\coloneqq\left(B(y_{r,j};\rho_{r,y_{r,j}})\cap\operatorname{cl}B(x;M)\right)\setminus\bigcup_{k<j}C_{r,k}\text{
    for }1\leq j\leq N_{r}.
  \]
  Then $\{C_{r,j}\}_{1\leq j\leq N_{r}}$ partitions $\operatorname{cl}B(x;M)$.
  For brevity, let $\beta_{r,j}\coloneqq\beta^{r,y_{r,j}}$ and note that
  \begin{equation}
    J(r,y^{\prime};a,\beta_{r,j}(a))\leq
    v^{+}(r,y^{\prime})+2\epsilon\text{ for all }y^{\prime}\in
    C_{r,j}\text{ and }a\in\mathcal{A}^{q,Q}(r).
    \label{eq:dpp_proof_upper_cell}
  \end{equation}

  Next, for each $a\in\mathcal{A}^{q,Q}(t)$, define
  \[
    \bar{\beta}(a)_{s}\coloneqq{} \beta(a)_{s}\boldsymbol{1}_{\{s\leq\theta_a\}}
    +\sum_{r\in
    R}\boldsymbol{1}_{\{\theta_a=r\}}\boldsymbol{1}_{\{s>r\}}\left[\boldsymbol{1}_{\{X^{a}_{r}\notin\operatorname{cl}B(x;M)\}}\beta(a)_{s}
      +\sum_{j=1}^{N_{r}}\boldsymbol{1}_{\{X^{a}_{r}\in
    C_{r,j}\}}\beta_{r,j}(a|_{(r,T]})_{s}\right].
  \]
  That is, up to and including $\theta_a$, the new control agrees with
  the old. After $\theta_a$, the new control agrees with the old if
  the stopped state does not fall into one of the Borelian
  cells. Otherwise, it uses one of the cell-specific continuations.
  Note that $\bar{\beta}$ is a strategy on the restricted control set
  $\mathcal{A}^{q,Q}(t)$ since it is
  \begin{itemize}
    \item strongly nonanticipative: it is constructed using strongly
      nonanticipative
      strategies and switches using a family of strongly
      nonanticipative stopping
      times. Rigorously, on the restricted domain, for
      $a,a^{\prime}\in\mathcal{A}^{q,Q}(t)$
      and $\tau\in\mathscr{T}_{t}([t,T])$ such that $a\equiv a^{\prime}$ on
      $\llbracket t,\tau\rrbracket$,
      we need to prove the equivalence of pasted controls up to $\tau$
      (i.e., $\bar{\beta}(a)\equiv\bar{\beta}(a^{\prime})$ on
      $\llbracket t,\tau\rrbracket$).
      Since the strategy $\beta$ is strongly nonanticipative, the baseline
      controls agree up to $\tau$ (i.e., $\beta(a)\equiv\beta(a^{\prime})$
      on $\llbracket t,\tau\rrbracket$). Similarly, since the family of
      stopping times is strongly nonanticipative,
      $\theta_a\equiv\theta_{a^\prime}$
      on $\llbracket t,\tau\rrbracket$. We split into cases:
      \begin{itemize}
        \item On $\{\theta_a>\tau\}$, both pasted controls agree with
          the baseline
          controls up to $\tau$.
        \item On $\{\theta_a\leq\tau\}$, the equivalence of the
          stopping times gives
          $\theta_a=\theta_{a^\prime}$. Uniqueness of the state
          process gives $X^{a}_{\theta_a}=X^{a^\prime}_{\theta_{a^\prime}}$
          and hence both stopped states are either outside
          $\operatorname{cl}B(x;M)$
          or inside the same cell $C_{r,j}$ on the event $\{\theta_a=r\}$. In
          the former case, both pasted controls once again agree with
          the baseline
          controls up to $\tau$. In the latter case, the restrictions
          $a|_{(r,T]}$
          and $a^{\prime}|_{(r,T]}$ agree on $\llbracket r,\tau\rrbracket$,
          so strong nonanticipativity of $\beta_{r,j}$ gives equality of the
          continuation controls on $\llbracket r,\tau\rrbracket$.
      \end{itemize}
    \item delayed: it is constructed from finitely many delayed strategies.
      The common refinement of their corresponding partitions, together
      with the finite set $R$, yields a partition for $\bar{\beta}$.
    \item an r-strategy: its definition is a finite pasting of r-strategies
      along events measurable at the corresponding deterministic switching
      times $r\in R$.
  \end{itemize}
  Note that $\bar{\beta}$ above is defined only on $\mathcal{A}^{q,Q}(t)$.
  For $a\in\mathcal{A}(t)\setminus\mathcal{A}^{q,Q}(t)$, let
  $a^\prime$ denote the truncation of $a$ obtained by keeping only
  the first $q-1$ impulses and stopping before the first impulse
  whose value is outside $Q$.
  Setting $\bar{\beta}(a)\coloneqq\bar{\beta}(a^\prime)$, this
  extends $\bar{\beta}$ to all of $\mathcal{A}(t)$.
  Since the truncation map is itself nonanticipative, the extension
  preserves the strategy properties verified above.
  Hence, $\bar{\beta}\in\mathscr{B}(t)$.

  For each $a\coloneqq(\tau_j, z_j)_j\in\mathcal{A}^{q,Q}(t)$, let
  \[
    J_{1}(t,x;a,\bar{\beta}(a);\theta_a)\coloneqq\int^{T}_{\theta_a}f(s,X^{t,x;a,\bar{\beta}(a)}_{s},\bar{\beta}(a)_{s})ds+\sum_{\theta_a<\tau_{j}\leq
    T}K(\tau_{j},z_{j})+g(X^{t,x;a,\bar{\beta}(a)}_{T}).
  \]
  Using \eqref{eq:dpp_proof_upper_cell} gives
  \begin{align*}
    \mathbb{E}\left[J_{1}(t,x;a,\bar{\beta}(a);\theta_a)\right] &
    \leq\mathbb{E}\left[\left(v^{+}(\theta_a,X^a_{\theta_a})+2\epsilon\right)\boldsymbol{1}_{\Omega\setminus
    E_{a}}+J_{1}(t,x;a,\bar{\beta}(a);\theta_a)\boldsymbol{1}_{E_{a}}\right]\\
    &
    =\mathbb{E}\left[v^{+}(\theta_a,X^a_{\theta_a})+2\epsilon\boldsymbol{1}_{\Omega\setminus
    E_{a}}+\left(J_{1}(t,x;a,\bar{\beta}(a);\theta_a)-v^{+}(\theta_a,X^a_{\theta_a})\right)\boldsymbol{1}_{E_{a}}\right]\\
    &
    \leq\mathbb{E}\left[v^{+}(\theta_a,X^a_{\theta_a})\right]+2\epsilon+\left(c+qK_{1}+\Vert
    v^{+}\Vert _{\infty}\right)\mathbb{P}(E_{a})\\
    & \leq\mathbb{E}\left[v^{+}(\theta_a,X^a_{\theta_a})\right]+3\epsilon.
  \end{align*}
  Therefore,
  \[
    J(t,x;a,\bar{\beta}(a))\leq\mathbb{E}\left[J_{0}(t,x;a,\beta(a);\theta_a)+v^{+}(\theta_a,X^a_{\theta_a})\right]+3\epsilon.
  \]
  Since $a$ was arbitrary, taking the supremum over $a\in\mathcal{A}^{q,Q}(t)$
  gives
  \[
    \sup_{a\in\mathcal{A}^{q,Q}(t)}J(t,x;a,\bar{\beta}(a))\leq\sup_{a\in\mathcal{A}^{q,Q}(t)}\mathbb{E}\left[J_{0}(t,x;a,\beta(a);\theta_a)+v^{+}(\theta_a,X_{\theta_a}^{a})\right]+3\epsilon.
  \]
  Since $\bar{\beta} \in \mathscr{B}(t)$,
  \[
    \adjustlimits\inf_{\beta^{\prime}\in\mathscr{B}(t)}\sup_{a\in\mathcal{A}^{q,Q}(t)}J(t,x;a,\beta^{\prime}(a))\leq\sup_{a\in\mathcal{A}^{q,Q}(t)}J(t,x;a,\bar{\beta}(a)).
  \]
  Combining the previous two inequalities,
  taking the infimum over the originally fixed $\beta\in\mathscr{B}(t)$,
  and using the arbitrariness of $\epsilon$, we obtain the fixed-$q$
  inequality. Taking the supremum over $q$ and applying
  \prettyref{lem:control_reduction} gives the desired inequality for
  $v^{+}$.
\end{proof}
We now prove the remaining inequality.
\begin{proof}[Proof of \prettyref{thm:dpp} ($\geq$)]
  Let $Q$ be the compact set in the statement of
  \prettyref{lem:control_reduction}.
  Fix $(t,x)\in\mathcal{O}$, a finite nonempty set $R\subset[t,T)$,
  and a strongly nonanticipative family of stopping times
  $\{\theta^{a,b}\}_{(a,b)\in\mathcal{A}(t)\times\mathcal{B}(t)}\subset\mathscr{T}_{t}(R)$.
  Let $\epsilon>0$ be arbitrary.

  Now, let $q$ be a positive integer and fix a strategy
  $\alpha\in\mathscr{A}^{q,Q}(t)$. Let $M$ be positive and, for each
  $b\in\mathcal{B}(t)$, let
  $X^{b}\coloneqq X^{t,x;\alpha(b),b}$,
  $\theta_b\coloneqq\theta^{\alpha(b),b}$, and
  $E_{b}\coloneqq\{|X^{b}_{\theta_b}-x|>M\}$.
  Arguing as in the previous proof, let $c$ and $K_{1}$ be
  defined as in \prettyref{lem:bounded_functional} and choose $M$
  sufficiently large so that
  \begin{equation}
    \left(c+qK_{1}+\Vert v^{-}\Vert
    _{\infty}\right)\sup_{b\in\mathcal{B}(t)}\mathbb{P}(E_{b})\leq\epsilon.
    \label{eq:dpp_proof_lower_tail}
  \end{equation}
  Again, arguing as in the previous proof, for each
  $r\in R$, choose a finite Borelian partition
  $\{C_{r,j}\}_{1\leq j\leq N_{r}}$ of $\operatorname{cl}B(x;M)$,
  positive integers $q_{r,j}$, and strategies
  $\alpha_{r,j}\in\mathscr{A}^{q_{r,j},Q}(r)$ such that
  \begin{equation}
    J(r,y^{\prime};\alpha_{r,j}(b),b)\geq
    v^{-}(r,y^{\prime})-2\epsilon\text{ for all }y^{\prime}\in
    C_{r,j}\text{ and }b\in\mathcal{B}(r).
    \label{eq:dpp_proof_lower_cell}
  \end{equation}

  Next, let $b\in\mathcal{B}(t)$ be arbitrary and let
  $\alpha(b)\coloneqq(\tau_{k},z_{k})_{k}$ and
  $N_{b}\coloneqq(\#\alpha(b))_{\theta_{b}}$. For each $r\in R$ and
  $1\leq j\leq N_{r}$, let
  \[
    \alpha_{r,j}(b|_{[r,T]})\coloneqq(\tau_{k}^{r,j},z_{k}^{r,j})_{k}.
  \]
  Subject to the convention $0\cdot\infty=0$, define
  $\bar{\alpha}(b)\coloneqq(\bar{\tau}_{k},\bar{z}_{k})_{k}$ by
  \[
    \bar{\tau}_{k}\coloneqq{} \tau_{k}\boldsymbol{1}_{\{k\leq N_{b}\}}
    +\sum_{r\in
    R}\boldsymbol{1}_{\{\theta_{b}=r\}}\boldsymbol{1}_{\{k>N_{b}\}}\left[\boldsymbol{1}_{\{X^{b}_{r}\notin\operatorname{cl}B(x;M)\}}\tau_{k}
      +\sum_{j=1}^{N_{r}}\boldsymbol{1}_{\{X^{b}_{r}\in
    C_{r,j}\}}\tau_{k-N_{b}}^{r,j}\right].
  \]
  The impulse $\bar{z}_{k}$ is defined analogously, using the same
  branch indicators.
  Note that this is analogous to the construction of $\bar{\beta}$ in
  the previous proof.
  For brevity, let
  \[
    \bar{q}\coloneqq q+\max_{r,j}q_{r,j}.
  \]
  Arguing as in the previous proof, $\bar{\alpha}$ is strongly
  nonanticipative, delayed, and an r-strategy. Moreover,
  $\bar{\alpha}$ takes values
  in $\mathcal{A}^{\bar{q},Q}(t)$ since the baseline control has less
  than $q$ impulses, each continuation has less than $q_{r,j}$ impulses,
  and all impulses lie in $Q$. Hence
  $\bar{\alpha}\in\mathscr{A}^{\bar{q},Q}(t)$.

  For each $b\in\mathcal{B}(t)$, let
  \[
    J_{1}(t,x;\bar{\alpha}(b),b;\theta_{b})\coloneqq\int^{T}_{\theta_{b}}f(s,X^{t,x;\bar{\alpha}(b),b}_{s},b_{s})ds+\sum_{\{j>N_{b}\colon\bar{\tau}_{j}\leq
    T\}}K(\bar{\tau}_{j},\bar{z}_{j})+g(X^{t,x;\bar{\alpha}(b),b}_{T}).
  \]
  Using \eqref{eq:dpp_proof_lower_cell} and arguing
  as in the previous proof,
  \[
    \mathbb{E}\left[J_{1}(t,x;\bar{\alpha}(b),b;\theta_b)\right]\geq\mathbb{E}\left[v^{-}(\theta_b,X_{\theta_b}^{b})\right]-3\epsilon.
  \]
  Consequently,
  \[
    J(t,x;\bar{\alpha}(b),b)\geq\mathbb{E}\left[J_{0}(t,x;\alpha(b),b;\theta_b)+v^{-}(\theta_b,X_{\theta_b}^{b})\right]-3\epsilon.
  \]
  Since $b$ was arbitrary, taking the infimum over $b\in\mathcal{B}(t)$
  gives
  \[
    \inf_{b\in\mathcal{B}(t)}J(t,x;\bar{\alpha}(b),b)\geq\inf_{b\in\mathcal{B}(t)}\mathbb{E}\left[J_{0}(t,x;\alpha(b),b;\theta_b)+v^{-}(\theta_b,X_{\theta_b}^{b})\right]-3\epsilon.
  \]
  Since $\bar{\alpha}\in\mathscr{A}^{\bar{q},Q}(t)$,
  \prettyref{lem:control_reduction}
  gives
  \[
    v^{-}(t,x)\geq\inf_{b\in\mathcal{B}(t)}J(t,x;\bar{\alpha}(b),b).
  \]
  Combining the previous two inequalities and using the arbitrariness
  of $q$, $\alpha$, and $\epsilon$ gives the desired inequality for
  $v^{-}$.
\end{proof}

\section{\label{sec:dpe_proof}Dynamic programming equation}

We prove, in this section, \prettyref{thm:dpe}. We first give a lemma,
which appears in slightly different flavours in \citep[Proposition
2.3]{MR1232083},
\citep[Lemma 5.1]{MR2284012}, \citep[Lemma 4.3]{MR2568293}, and
possibly elsewhere. We provide a proof since our setting is slightly
different from the aforementioned.
\begin{lem}
  \label{lem:intervention_results} Let
  $u,w:\operatorname{cl}\mathcal{O}\rightarrow\mathbb{R}$
  be bounded. $\mathcal{M}$ is monotone: if $u\geq w$ pointwise,
  $\mathcal{M}u\geq\mathcal{M}w$
  pointwise. Moreover, $\mathcal{M}u_{*}$ (resp. $\mathcal{M}u^{*}$)
  is lower (resp. upper) semicontinuous and
  $\mathcal{M}u_{*}\leq(\mathcal{M}u)_{*}$
  (resp. $(\mathcal{M}u)^{*}\leq\mathcal{M}u^{*}$).
\end{lem}
\begin{proof}
  The monotonicity property follows directly from the definition.

  Let $\epsilon>0$ and let $(t_{n},x_{n})_{n}$ be a
  $\operatorname{cl}\mathcal{O}$-valued
  sequence converging to some $(t,x)$. Pick $z^{\epsilon}\in Z$ such
  that
  $u_{*}(t,x+\Gamma(t,z^{\epsilon}))+K(t,z^{\epsilon})+\epsilon\geq\mathcal{M}u_{*}(t,x)$.
  Then,
  \begin{align*}
    \liminf_{n\rightarrow\infty}\mathcal{M}u_{*}(t_{n},x_{n}) &
    \geq\liminf_{n\rightarrow\infty}\left\{
    u_{*}(t_{n},x_{n}+\Gamma(t_{n},z^{\epsilon}))+K(t_{n},z^{\epsilon})\right\}
    \\
    & \geq
    u_{*}(t,x+\Gamma(t,z^{\epsilon}))+K(t,z^{\epsilon})\geq\mathcal{M}u_{*}(t,x)-\epsilon.
  \end{align*}
  Since $\epsilon$ was arbitrary, it follows that $\mathcal{M}u_{*}$
  is lower semicontinuous. By monotonicity, we have
  $\mathcal{M}u\geq\mathcal{M}u_{*}$
  pointwise, so that we can take lower semicontinuous envelopes of both
  sides to get $(\mathcal{M}u)_{*}\geq(\mathcal{M}u_{*})_{*}=\mathcal{M}u_{*}$
  as desired.

  Lastly, let $(t_{n},x_{n})_{n}$ be a $\operatorname{cl}\mathcal{O}$-valued
  sequence converging to some $(t,x)$. Due to the upper semicontinuity
  and boundedness of $u^{*}$ and the growth conditions on $K$
  (\prettyref{assu:impulse_assumptions}
  \ref{enu:negative_growth_condition}), there exists a compact set
  $Q\subset Z$ such that for each $n$, there is a $z_{n}\in Q$ with
  $\mathcal{M}u^{*}(t_{n},x_{n})=u^{*}(t_{n},x_{n}+\Gamma(t_{n},z_{n}))+K(t_{n},z_{n})$.
  Therefore, $(z_{n})_{n}$ admits a convergent subsequence with limit
  $\hat{z}\in Z$ and hence
  \begin{align*}
    \limsup_{n\rightarrow\infty}\mathcal{M}u^{*}(t_{n},x_{n}) &
    =\limsup_{n\rightarrow\infty}\left\{
    u^{*}(t_{n},x_{n}+\Gamma(t_{n},z_{n}))+K(t_{n},z_{n})\right\} \\
    & \leq u^{*}(t,x+\Gamma(t,\hat{z}))+K(t,\hat{z})\leq\mathcal{M}u^{*}(t,x).
  \end{align*}
  The rest of the proof is similar to previous arguments.
\end{proof}
We are now ready to prove \prettyref{thm:dpe}. The boundary conditions
follow from \prettyref{assu:dpp_and_dpe_assumptions}
\ref{enu:enforce_boundary_conditions}
(cf. \citep[Theorem 4.2]{MR2568293}), and are thus ignored in the
proof. We first show that $v^{+}$ is a subsolution.
\begin{proof}[Proof of \prettyref{thm:dpe} (subsolution)]
  Let $(t,x,\varphi)\in\mathcal{O}\times C^{1,2}(\mathcal{O})$ be
  such that $((v^{+})^{*}-\varphi)(t,x)=0$ is a \emph{strict} local
  maximum of $(v^{+})^{*}-\varphi$. By
  \prettyref{lem:compact_test_functions} of
  \prettyref{app:equivalent_definitions},
  we can assume $\varphi$ is compactly supported. To prove the result,
  we assume
  \[
    \inf_{b\in B}\left\{
    (\partial_{t}+\text{\L}^{b})\varphi(t,x)+f(t,x,b)\right\}
    <0\text{ and }((v^{+})^{*}-\mathcal{M}(v^{+})^{*})(t,x)>0
  \]
  and show that this contradicts the first inequality appearing in the
  DPP (\prettyref{thm:dpp}). The above implies that for some $\hat{b}\in B$,
  \[
    (\partial_{t}+\text{\L}^{\hat{b}})\varphi(t,x)+f(t,x,\hat{b})<0.
  \]
  For each $r>0$, define the set
  \begin{equation}
    \mathcal{N}_{r}\coloneqq((t-r,t+r)\times
    B(x;r))\cap\mathcal{O}.\label{eq:N_r}
  \end{equation}
  By \prettyref{lem:intervention_results}, $\varphi-\mathcal{M}(v^{+})^{*}$
  is lower semicontinuous. Since for some $\delta>0$,
  \[
    (\varphi-\mathcal{M}(v^{+})^{*})(t,x)=((v^{+})^{*}-\mathcal{M}(v^{+})^{*})(t,x)\geq4\delta,
  \]
  it follows that we can find $h>0$ (by lower semicontinuity) such
  that $t+2h<T$ and the following claims hold on the set
  $\operatorname{cl}\mathcal{N}_{2h}$:
  \begin{gather*}
    ((v^{+})^{*}-\varphi)(t,x)=0\text{ is a strict maximum of
    }(v^{+})^{*}-\varphi\text{,}\\
    \varphi-\mathcal{M}(v^{+})^{*}\geq3\delta\text{, and
    }(\partial_{t}+\text{\L}^{\hat{b}})\varphi+f(\cdot,\cdot,\hat{b})\leq0.
  \end{gather*}
  Since $(t,x)$ is a strict maximum point of $(v^{+})^{*}-\varphi$,
  \[
    -3\gamma\coloneqq\max_{\operatorname{cl}\mathcal{N}_{2h}\setminus\operatorname{int}\mathcal{N}_{h/2}}((v^{+})^{*}-\varphi)<0.
  \]
  Let $(t_{n},x_{n})_{n}$ be a sequence converging to $(t,x)$ with
  $v^{+}(t_{n},x_{n})\rightarrow(v^{+})^{*}(t,x)$. Further let
  \[
    \eta_{n}\coloneqq\varphi(t_{n},x_{n})-v^{+}(t_{n},x_{n})\geq0
  \]
  and note that $\eta_{n}\rightarrow0$ as $n\rightarrow\infty$.

  Now, choose $n$ large enough such that $(t_{n},x_{n})\in\mathcal{N}_{h}$
  and
  \[
    \left(\gamma\wedge\delta\right)-\eta_{n}>0.
  \]
  By \prettyref{lem:control_reduction}, there exists a nonempty
  compact set $Q\subset Z$ such that, since
  $\left(\gamma\wedge\delta\right)-\eta_{n}>0$, we can choose an
  integer $q\geq1$ satisfying
  \begin{equation}
    v^{+}(t_{n},x_{n})\leq\frac{\left(\gamma\wedge\delta\right)-\eta_{n}}{4}+\adjustlimits\inf_{\beta\in\mathscr{B}(t_{n})}\sup_{a\in\mathcal{A}^{q,Q}(t_{n})}J(t_{n},x_{n};a,\beta(a)).
    \label{eq:dpe_proof_subsolution_q}
  \end{equation}
  With a slight abuse of notation, we also use $\hat{b}$ (chosen in
  the previous paragraph) to refer to a constant control in $\mathcal{B}(t_{n})$
  taking on the value $\hat{b}\in B$. Let $\hat{a}\in\mathcal{A}(t_{n})$
  denote a control with no impulses (i.e., $(\#\hat{a})_{T}=0$). Further
  let $Y\coloneqq X^{t_{n},x_{n};\hat{a},\hat{b}}$ and, for $r>0$,
  \[
    \psi^{r}\coloneqq\inf\left\{
    s>t_{n}\colon(s,Y_{s})\notin\mathcal{N}_{r}\right\} .
  \]
  For each positive integer $m$, also let
  \[
    \Delta_{m}\coloneqq\frac{t+2h-t_{n}}{2^{m}}\quad\text{and}\quad\Pi_{m}\coloneqq\left\{
    t_{n}+k\Delta_{m}\colon k=0,\ldots,2^{m}\right\}
  \]
  and define
  \[
    \psi_{m}\coloneqq\min(\Pi_{m}\cap[\psi^{h},t+2h])\text{ and
    }A_{m}\coloneqq\{\psi_{m}\leq\psi^{2h}\}\cap\{(\psi_{m},Y_{\psi_{m}})\notin\mathcal{N}_{h/2}\}.
  \]
  Note, in particular, that the stopping time $\psi_{m}$ takes values
  in the finite set $\Pi_{m}$.
  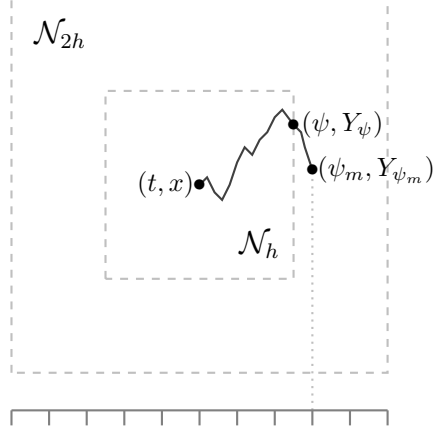
\begin{figure}
    \begin{center}
      \begin{tikzpicture}
        \draw [black!25, thick, dotted] (1.5,-3) -- (1.5,0.2);
        \draw [black!50, thick] (-2.5,-3) -- coordinate (x axis mid) (2.5,-3);
        \foreach \x in {-2.5,-2.0,...,2.5} \draw [black!50, thick]
        (\x,-3) -- (\x,-3.2);
        \draw [black!25, dashed, thick] (-2.5,-2.5) rectangle (2.5,2.5);
        \node at (-1.85,2) {$\mathcal{N}_{2h}$};
        \draw [black!25, dashed, thick] (-1.25,-1.25) rectangle (1.25,1.25);
        \node at (0.8,-0.8) {$\mathcal{N}_{h}$};
        \draw [black!75, thick]  (0,0) -- (0.1,0.1) -- (0.2,-0.1) --
        (0.3,-0.2) -- (0.4,0) -- (0.5,0.3) -- (0.6,0.5) -- (0.7,0.4)
        -- (0.8,0.6) -- (0.9,0.7) -- (1.0,0.9) -- (1.1,1.0) --
        (1.25,0.8) -- (1.35,0.7) -- (1.4,0.5) -- (1.5,0.2);
        \node at (0, 0) (center) {\footnotesize \textbullet};
        \node [left of=center, node distance=0.45cm] {\footnotesize $(t,x)$};
        \node at (1.25, 0.8) (exit) {\footnotesize \textbullet};
        \node [right of=exit, node distance=0.84cm] {\footnotesize
        $(\psi^{h},Y_{\psi^{h}})$};
        \node at (1.5,0.2) (approximate) {\footnotesize \textbullet};
        \node [right of=approximate, node distance=0.85cm]
        {\footnotesize $(\psi_m,Y_{\psi_m})$};
      \end{tikzpicture}
    \end{center}

    \caption{Approximating $\psi^{h}$ by $\psi_{m}$}
  \end{figure}

  Now, for each $a\coloneqq(\tau_{j},z_{j})_{j\geq1}\in\mathcal{A}(t_{n})$
  and $b\in\mathcal{B}(t_{n})$, define $\bar{\tau}_{m}$ by
  \[
    \bar{\tau}_{m}\coloneqq\min(\Pi_{m}\cap[\tau_{1},t+2h])\mathbf{1}_{\{\tau_{1}\leq
    t+2h\}}+\left(\infty\right)\mathbf{1}_{\{\tau_{1}>t+2h\}}
  \]
  and set $\theta^{a,b}\coloneqq\bar{\tau}_{m}\wedge\psi_{m}$.
  It follows that
  $\{\theta^{a,b}\}_{(a,b)\in\mathcal{A}(t_{n})\times\mathcal{B}(t_{n})}$
  is a strongly nonanticipative family of stopping times taking values
  in the finite set $\Pi_{m}$. Recall that $f$, $(v^{+})^{*}$, and
  $\mathcal{M}(v^{+})^{*}$ are bounded, $\varphi$ has compact support,
  and $\boldsymbol{1}_{\Omega\setminus A_{m}}\rightarrow0$
  $\mathbb{P}$-almost surely as $m\rightarrow\infty$. As such, we
  can choose $m$ large enough such that for all
  $a\in\mathcal{A}(t_{n})$, writing $\theta\coloneqq\theta^{a,\hat{b}}$,
  \[
    \mathbb{E}\left[\left(\int_{t_{n}}^{\theta}-(\partial_{t}+\text{\L}^{\hat{b}})\varphi(s,Y_{s})ds+\varphi(\theta,Y_{\theta})\right)\boldsymbol{1}_{\Omega\setminus
    A_{m}}\right]\geq-\left(\gamma\wedge\delta\right)
  \]
  and
  \[
    0\geq\mathbb{E}\left[\left(
        \begin{gathered}\int_{t_{n}}^{\theta}f(s,Y_{s},\hat{b})ds+(v^{+})^{*}(\theta,Y_{\theta})\boldsymbol{1}_{\{\theta<\bar{\tau}_{m}\}}\\
          +\mathcal{M}(v^{+})^{*}(\theta,Y_{\theta})\boldsymbol{1}_{\{\theta=\bar{\tau}_{m}\}}+3\left(\gamma\wedge\delta\right)
        \end{gathered}
      \right)\boldsymbol{1}_{\Omega\setminus
    A_{m}}\right]-\left(\gamma\wedge\delta\right).
  \]
  Dynkin's formula and the inequalities above imply that, for each
  $a\coloneqq(\tau_{j},z_{j})_{j\geq1}\in\mathcal{A}(t_{n})$
  \begin{align*}
    v^{+}(t_{n},x_{n})+\eta_{n} & =\varphi(t_{n},x_{n})\\
    &
    =\mathbb{E}\left[\left(\int_{t_{n}}^{\theta}-(\partial_{t}+\text{\L}^{\hat{b}})\varphi(s,Y_{s})ds+\varphi(\theta,Y_{\theta})\right)\left(\boldsymbol{1}_{A_{m}}+\boldsymbol{1}_{\Omega\setminus
    A_{m}}\right)\right]\\
    & \geq\mathbb{E}\left[\left(
        \begin{gathered}\int_{t_{n}}^{\theta}f(s,Y_{s},\hat{b})ds+(v^{+})^{*}(\theta,Y_{\theta})\boldsymbol{1}_{\{\theta<\bar{\tau}_{m}\}}\\
          +\mathcal{M}(v^{+})^{*}(\theta,Y_{\theta})\boldsymbol{1}_{\{\theta=\bar{\tau}_{m}\}}+3\left(\gamma\wedge\delta\right)
        \end{gathered}
    \right)\boldsymbol{1}_{A_{m}}\right]-\left(\gamma\wedge\delta\right)\\
    & \geq\mathbb{E}\left[
      \begin{gathered}\int_{t_{n}}^{\theta}f(s,Y_{s},\hat{b})ds+(v^{+})^{*}(\theta,Y_{\theta})\boldsymbol{1}_{\{\theta<\bar{\tau}_{m}\}}\\
        +\mathcal{M}(v^{+})^{*}(\theta,Y_{\theta})\boldsymbol{1}_{\{\theta=\bar{\tau}_{m}\}}
      \end{gathered}
    \right]+\left(\gamma\wedge\delta\right)
  \end{align*}
  where it is understood that $\theta\coloneqq\theta^{a,\hat{b}}$.
  For the remainder of the proof, we will omit the superscript in
  $\theta^{a,\hat{b}}$
  in the same manner for brevity. Taking supremums in the above over
  $a\in\mathcal{A}^{q,Q}(t_{n})$, we get
  \begin{equation}
    v^{+}(t_{n},x_{n})\geq(\gamma\wedge\delta)-\eta_{n}+\sup_{a\in\mathcal{A}^{q,Q}(t_{n})}\mathbb{E}\left[
      \begin{gathered}\int_{t_{n}}^{\theta}f(s,Y_{s},\hat{b})ds+(v^{+})^{*}(\theta,Y_{\theta})\boldsymbol{1}_{\{\theta<\bar{\tau}_{m}\}}\\
        +\mathcal{M}(v^{+})^{*}(\theta,Y_{\theta})\boldsymbol{1}_{\{\theta=\bar{\tau}_{m}\}}
      \end{gathered}
    \right].\label{eq:dpe_proof_subsolution_y_payoff}
  \end{equation}
  Let $X^{a}\coloneqq X^{t_{n},x_{n};a,\hat{b}}$. On
  $\{\theta<\bar{\tau}_{m}\}$, we have $\theta<\tau_{1}$, so no impulse
  has occurred before $\theta$. Hence $Y_{\theta}=X_{\theta}^{a}$ and
  $\sum_{\tau_{j}\leq\theta}K(\tau_{j},z_{j})=0$. On
  $\{\theta=\bar{\tau}_{m}\}$, the interval $[\tau_{1},\theta]$ has
  length at most $\Delta_{m}$. If all impulses in this short interval
  are viewed as occurring at time $\theta$, then repeated use of
  \prettyref{assu:dpp_and_dpe_assumptions}
  \ref{enu:multiple_impulses_are_suboptimal} collapses them into a single
  impulse. Hence, on this event,
  \[
    \mathcal{M}(v^{+})^{*}(\theta,Y_{\theta})\geq(v^{+})^{*}\left(\theta,Y_{\theta}+\sum_{\tau_{j}\leq\theta}\Gamma(\theta,z_{j})\right)+\sum_{\tau_{j}\leq\theta}K(\theta,z_{j}).
  \]
  Since the supremum is over $a\in\mathcal{A}^{q,Q}(t_{n})$, the uniform
  continuity of $K$ and $\Gamma$ on $[0,T]\times Q$, the fixed-time Lipschitz
  continuity of $(v^{+})^{*}$, the boundedness of $f$, and standard
  estimates for the SDE imply that the actual stopped payoff is bounded
  above by this idealized collapsed payoff up to an error
  $\epsilon_{m}\downarrow0$, uniformly over $a\in\mathcal{A}^{q,Q}(t_{n})$;
  namely,
  \begin{multline*}
    \sup_{a\in\mathcal{A}^{q,Q}(t_{n})}\mathbb{E}\left[J_{0}(t_{n},x_{n};a,\hat{b};\theta)+(v^{+})^{*}(\theta,X_{\theta}^{a})\right]\\
    \leq\epsilon_{m}+\sup_{a\in\mathcal{A}^{q,Q}(t_{n})}\mathbb{E}\left[
      \begin{gathered}\int_{t_{n}}^{\theta}f(s,Y_{s},\hat{b})ds+(v^{+})^{*}(\theta,Y_{\theta})\boldsymbol{1}_{\{\theta<\bar{\tau}_{m}\}}\\
        +\mathcal{M}(v^{+})^{*}(\theta,Y_{\theta})\boldsymbol{1}_{\{\theta=\bar{\tau}_{m}\}}
      \end{gathered}
    \right].
  \end{multline*}
  Combining this with \eqref{eq:dpe_proof_subsolution_y_payoff} and
  taking $m$ sufficiently large,
  \begin{equation}
    v^{+}(t_{n},x_{n})\geq\frac{(\gamma\wedge\delta)-\eta_{n}}{2}+\sup_{a\in\mathcal{A}^{q,Q}(t_{n})}\mathbb{E}\left[J_{0}(t_{n},x_{n};a,\hat{b};\theta)+(v^{+})^{*}(\theta,X_{\theta}^{a})\right].\label{eq:dpe_proof_1}
  \end{equation}
  Since $v^{+}\leq(v^{+})^{*}$, \eqref{eq:dpe_proof_subsolution_q}
  and the fixed-$q$ part of the DPP give
  \[
    v^{+}(t_{n},x_{n})\leq\frac{(\gamma\wedge\delta)-\eta_{n}}{4}+\sup_{a\in\mathcal{A}^{q,Q}(t_{n})}\mathbb{E}\left[J_{0}(t_{n},x_{n};a,\hat{b};\theta)+(v^{+})^{*}(\theta,X_{\theta}^{a})\right],
  \]
  contradicting \eqref{eq:dpe_proof_1}.
\end{proof}
We now show that $v^{-}$ is a supersolution.
\begin{proof}[Proof of \prettyref{thm:dpe} (supersolution)]
  Let $(t,x)\in\operatorname{cl}\mathcal{O}$, $z\in Z$, and
  $\epsilon>0$ be arbitrary.
  Set $y\coloneqq x+\Gamma(t,z)$.
  By the definition of $v^{-}$, we can choose an integer $q\geq1$
  and a strategy $\alpha\in\mathscr{A}^{q}(t,y)$ such that
  \[
    \inf_{b\in\mathcal{B}(t)}J(t,y;\alpha(b),b)\geq v^{-}(t,y)-\epsilon.
  \]
  Define a new strategy $\bar{\alpha}$ at $(t,x)$ by applying the
  impulse $z$ immediately at time $t$, and then following $\alpha$.
  More precisely, if $\alpha(b)=(\tau_{j},z_{j})_{j\geq1}$, then
  $\bar{\alpha}(b)$ is the impulse control obtained by prepending
  $(t,z)$ to this sequence. Since $\alpha(b)$ has fewer than $q$ impulses,
  $\bar{\alpha}(b)$ has fewer than $q+1$ impulses, and the r-strategy
  property is preserved by this deterministic prepending. Thus
  $\bar{\alpha}\in\mathscr{A}^{q+1}(t,x)$. Moreover, for every
  $b\in\mathcal{B}(t)$,
  \[
    J(t,x;\bar{\alpha}(b),b)=K(t,z)+J(t,y;\alpha(b),b).
  \]
  Therefore,
  \[
    v^{-}(t,x) \geq\inf_{b\in\mathcal{B}(t)}J(t,x;\bar{\alpha}(b),b)
    =K(t,z)+\inf_{b\in\mathcal{B}(t)}J(t,y;\alpha(b),b)
    \geq K(t,z)+v^{-}(t,y)-\epsilon.
  \]
  Letting $\epsilon\downarrow0$, we get
  \[
    v^{-}(t,x)\geq K(t,z)+v^{-}(t,x+\Gamma(t,z)).
  \]
  Since $z\in Z$ was arbitrary, $v^{-}\geq\mathcal{M}v^{-}$ on
  $\operatorname{cl}\mathcal{O}$. By the monotonicity of $\mathcal{M}$,
  $\mathcal{M}v^{-}\geq\mathcal{M}(v^{-})_{*}$.
  Since $\mathcal{M}(v^{-})_{*}$ is lower semicontinuous by
  \prettyref{lem:intervention_results},
  $(v^{-})_{*}\geq\mathcal{M}(v^{-})_{*}$ on $\operatorname{cl}\mathcal{O}$.

  Next, let $(t,x,\varphi)\in\mathcal{O}\times C^{1,2}(\mathcal{O})$ be such
  that $((v^{-})_{*}-\varphi)(t,x)=0$ is a \emph{strict} local minimum
  of $(v^{-})_{*}-\varphi$. By \prettyref{lem:compact_test_functions}
  of \prettyref{app:equivalent_definitions},
  we can assume $\varphi$ is compactly supported. To complete the proof,
  we assume
  \[
    \inf_{b\in B}\left\{
    (\partial_{t}+\text{\L}^{b})\varphi(t,x)+f(t,x,b)\right\} >0
  \]
  and show that this contradicts the second inequality appearing in
  the DPP (\prettyref{thm:dpp}). For each $r>0$, define $\mathcal{N}_{r}$
  as in \eqref{eq:N_r}. By continuity, we can find $h>0$ such that
  $t+2h<T$ and the following claims hold on the set
  $\operatorname{cl}\mathcal{N}_{2h}$:
  \begin{gather*}
    ((v^{-})_{*}-\varphi)(t,x)=0\text{ is a strict minimum of
    }(v^{-})_{*}-\varphi\\
    \text{ and }\inf_{b\in B}\left\{
    (\partial_{t}+\text{\L}^{b})\varphi+f(\cdot,\cdot,b)\right\} \geq0.
  \end{gather*}
  Since $(t,x)$ is a strict minimum point of $(v^{-})_{*}-\varphi$,
  \[
    3\gamma\coloneqq\min_{\operatorname{cl}\mathcal{N}_{2h}\setminus\operatorname{int}\mathcal{N}_{h/2}}((v^{-})_{*}-\varphi)>0.
  \]
  Let $(t_{n},x_{n})_{n}$ be a sequence converging to $(t,x)$ with
  $v^{-}(t_{n},x_{n})\rightarrow(v^{-})_{*}(t,x)$. Further let
  \[
    \eta_{n}\coloneqq v^{-}(t_{n},x_{n})-\varphi(t_{n},x_{n})\geq0
  \]
  and note that $\eta_{n}\rightarrow0$ as $n\rightarrow\infty$.

  Now, choose $n$ large enough such that $(t_{n},x_{n})\in\mathcal{N}_{h}$
  and
  \[
    \eta_{n}-\gamma<0.
  \]
  Let $\hat{a}\in\mathcal{A}(t_{n})$ denote a control with no impulses
  (i.e., $(\#\hat{a})_{T}=0$). For each $b\in\mathcal{B}(t_{n})$,
  let $X^{b}\coloneqq X^{t_{n},x_{n};\hat{a},b}$ and, for $r>0$,
  \[
    \psi^{b,r}\coloneqq\inf\left\{
    s>t_{n}\colon(s,X_{s}^{b})\notin\mathcal{N}_{r}\right\} .
  \]
  For each $b\in\mathcal{B}(t_{n})$ and positive integer $m$, let
  $\Delta_{m}$ and $\Pi_{m}$ be as in the previous proof and
  \[
    \psi_{m}^{b}\coloneqq\min(\Pi_{m}\cap[\psi^{b,h},t+2h])\text{ and
    }A_{m}^{b}\coloneqq\{\psi_{m}^{b}\leq\psi^{b,2h}\}\cap\{(\psi_{m}^{b},X_{\psi_{m}^{b}}^{b})\notin\mathcal{N}_{h/2}\}.
  \]

  By construction,
  \[
    \Omega\setminus A_m^b\subseteq\left\{
      \sup \left\{
        \left|X_s^b - X_u^{b}\right|
        \colon s, u \in \left[t_n, t + 2h\right], \left|s-u\right| \leq \Delta_m
      \right\}
      \geq \frac{h}{2}
    \right\}.
  \]
  Thus, by standard SDE modulus estimates and compactness of $B$,
  \[
    \adjustlimits\lim_{m\rightarrow\infty}\sup_{b\in\mathcal{B}(t_{n})}\mathbb{P}(\Omega\setminus
    A_{m}^{b})=0.
  \]
  Since $f$, $(v^{-})_{*}$, and $\varphi$ are bounded, and since
  $(\partial_{t}+\text{\L}^{b})\varphi$ is bounded uniformly over
  $b\in B$, we can choose $m$ large enough so that for each
  $b\in\mathcal{B}(t_{n})$,
  \[
    \mathbb{E}\left[\left(\int_{t_{n}}^{\psi_{m}^{b}}-(\partial_{t}+\text{\L}^{b_{s}})\varphi(s,X_{s}^{b})ds+\varphi(\psi_{m}^{b},X_{\psi_{m}^{b}}^{b})\right)\boldsymbol{1}_{\Omega\setminus
    A_{m}^{b}}\right]\leq\gamma,
  \]
  and
  \[
    0\leq\mathbb{E}\left[\left(\int_{t_{n}}^{\psi_{m}^{b}}f(s,X_{s}^{b},b_{s})ds+(v^{-})_{*}(\psi_{m}^{b},X_{\psi_{m}^{b}}^{b})-3\gamma\right)\boldsymbol{1}_{\Omega\setminus
    A_{m}^{b}}\right]+\gamma.
  \]
  Now, for each $a\in\mathcal{A}(t_{n})$ and $b\in\mathcal{B}(t_{n})$,
  define $\theta^{a,b}\coloneqq\psi_{m}^{b}$. It follows that
  $\{\theta^{a,b}\}_{(a,b)\in\mathcal{A}(t_{n})\times\mathcal{B}(t_{n})}$
  is a strongly nonanticipative family of stopping times taking values
  in the finite set $\Pi_{m}$. Dynkin's formula and the above inequalities
  imply that, for each $b\in\mathcal{B}(t_{n})$,
  \begin{align*}
    v^{-}(t_{n},x_{n})-\eta_{n} & =\varphi(t_{n},x_{n})\\
    &
    =\mathbb{E}\left[\left(\int_{t_{n}}^{\theta}-(\partial_{t}+\text{\L}^{b_{s}})\varphi(s,X_{s}^{b})ds+\varphi(\theta,X_{\theta}^{b})\right)\left(\boldsymbol{1}_{A_{m}^{b}}+\boldsymbol{1}_{\Omega\setminus
    A_{m}^{b}}\right)\right]\\
    &
    \leq\mathbb{E}\left[\left(\int_{t_{n}}^{\theta}f(s,X_{s}^{b},b_{s})ds+(v^{-})_{*}(\theta,X_{\theta}^{b})-3\gamma\right)\boldsymbol{1}_{A_{m}^{b}}\right]+\gamma\\
    &
    \leq\mathbb{E}\left[\int_{t_{n}}^{\theta}f(s,X_{s}^{b},b_{s})ds+(v^{-})_{*}(\theta,X_{\theta}^{b})\right]-\gamma
  \end{align*}
  where it is understood that $\theta\coloneqq\theta^{\hat{a},b}$.
  For the remainder of the proof, we will omit the superscript in
  $\theta^{\hat{a},b}$
  in the same manner for brevity. Taking infimums, we get
  \[
    v^{-}(t_{n},x_{n})\leq\eta_{n}-\gamma+\inf_{b\in\mathcal{B}(t_{n})}\mathbb{E}\left[\int_{t_{n}}^{\theta}f(s,X_{s}^{b},b_{s})ds+(v^{-})_{*}(\theta,X_{\theta}^{b})\right].
  \]
  However, the DPP and the inequality $v^{-}\geq(v^{-})_{*}$ imply that
  \[
    v^{-}(t_{n},x_{n})\geq\inf_{b\in\mathcal{B}(t_{n})}\mathbb{E}\left[\int_{t_{n}}^{\theta}f(s,X_{s}^{b},b_{s})ds+(v^{-})_{*}(\theta,X_{\theta}^{b})\right],
  \]
  a contradiction.
\end{proof}
As promised, we give below an example in which
\prettyref{assu:dpp_and_dpe_assumptions}
\ref{enu:enforce_boundary_conditions} is satisfied.
\begin{example}
  \label{exa:terminal_continuity}If $g\geq\mathcal{M}g$ pointwise,
  it follows that $v^{\pm}(T,\cdot)=g$ pointwise (cf.
  \citep{belak2016utility,MR3615458})
  so that \prettyref{assu:dpp_and_dpe_assumptions}
  \ref{enu:enforce_boundary_conditions}
  is satisfied. To see this, suppose we can find a point $x$ at which
  $v^{\pm}(T,x)\neq g(x)$. Then, by the definition of the value functions
  and \prettyref{assu:dpp_and_dpe_assumptions}
  \ref{enu:multiple_impulses_are_suboptimal},
  \[
    g(x)<v^{\pm}(T,x)=\sup_{z\in Z}\left\{
    K(T,z)+g(x+\Gamma(T,z))\right\} =\mathcal{M}g(x),
  \]
  a contradiction.
\end{example}

\section{\label{sec:comparison_principle_proof}Comparison principle}

We prove, in this section, \prettyref{thm:comparison_principle}.
We first prepare a few lemmas. The result below, which also appears
in \citep[Lemma 5.5]{MR2568293}, follows directly from sup-manipulations.
\begin{lem}
  \label{lem:more_intervention_results}Let
  $u,w:\operatorname{cl}\mathcal{O}\rightarrow\mathbb{R}$
  be bounded. $\mathcal{M}$ is convex:
  \[
    \mathcal{M}(\lambda
    u+\left(1-\lambda\right)w)\leq\lambda\mathcal{M}u+\left(1-\lambda\right)\mathcal{M}w\text{
    for }0\leq\lambda\leq1.
  \]
\end{lem}
We now perform a change of variables to introduce a positive ``discount''
term $\rho>0$. Concretely, let
\begin{equation}
  F_{\rho}(\cdot,u,Du(\cdot),D^{2}u(\cdot))\coloneqq
  \begin{cases}
    \min\{-\inf_{b\in
    B}\{(\partial_{t}+\text{\L}^{b}-\rho)u+f_{\rho}^{b}\},u-\mathcal{M}_{\rho}u\}
    & \text{on }\mathcal{O}\\
    \min\{u-g_{\rho},u-\mathcal{M}_{\rho}u\} & \text{on }\partial^{+}\mathcal{O}
  \end{cases}\label{eq:hjbi_discounted}
\end{equation}
where $f_{\rho}^{b}(t,x)\coloneqq f_{\rho}(t,x,b)\coloneqq e^{\rho t}f(t,x,b)$,
$g_{\rho}(x)\coloneqq e^{\rho T}g(x)$, $K_{\rho}(t,z)\coloneqq
e^{\rho t}K(t,z)$,
and
\[
  \mathcal{M}_{\rho}u(t,x)\coloneqq\sup_{z\in Z}\left\{
  u(t,x+\Gamma(t,z))+K_{\rho}(t,z)\right\} .
\]
The notion of viscosity solution for $F_{\rho}=0$ is identical to
that of $F=0$ (see \prettyref{def:viscosity_solution}).

Note that if $u$ is a subsolution of $F=0$, $e^{\rho t}u$ is a
subsolution of $F_{\rho}=0$. The same claim holds for supersolutions.
Therefore, we need only consider uniqueness under some $\rho>0$,
which we pick arbitrarily and leave fixed for the remainder of this
section. Note that Lemmas \ref{lem:intervention_results} and
\ref{lem:more_intervention_results}
remain valid for $\mathcal{M}_{\rho}$.

The following lemma allows us to construct a family of ``strict''
supersolutions of $F_{\rho}=0$ by taking combinations of an ordinary
supersolution and a specific constant. This technique appears in
\citep[Lemma 3.2]{MR1232083},
and is needed due to the implicit form of the obstacle.
\begin{lem}
  \label{lem:strict_supersolutions} Let $w$ be a supersolution of
  \eqref{eq:hjbi_discounted}, $c\coloneqq\max\{(\Vert
  f_{\rho}\Vert_{\infty}+1)/\rho,\Vert g_{\rho}\Vert_{\infty}+1\}$,
  and $\xi\coloneqq\min\{1,K_{0}\}$. Then, for each $0<\lambda<1$,
  $w_{\lambda}\coloneqq(1-\lambda)w+\lambda c$ is a supersolution of
  \begin{equation}
    F_{\rho}(\cdot,u,Du(\cdot),D^{2}u(\cdot))-\lambda\xi=0\text{ on
    }\operatorname{cl}\mathcal{O}.\label{eq:translated_hjbi}
  \end{equation}
\end{lem}
\begin{proof}
  Below, we treat $c$ both as a constant and a constant function on
  $\operatorname{cl}\mathcal{O}$ taking the value $c$. First, note
  that for $(t,x)\in\mathcal{O}$,
  \begin{multline*}
    \min\left\{ -\inf_{b\in B}\left\{
      (\partial_{t}+\text{\L}^{b}-\rho)c(t,x)+f_{\rho}(t,x,b)\right\}
    ,(c-\mathcal{M}_{\rho}c)(t,x)\right\} \\
    \geq\min\left\{ \rho c-\left\Vert f_{\rho}\right\Vert
    _{\infty},K_{0}\right\} \geq\xi.
  \end{multline*}
  Similarly, for $(t,x)\in\partial^{+}\mathcal{O}$,
  \[
    \min\left\{
    c(t,x)-g_{\rho}(x),(c-\mathcal{M}_{\rho}c)(t,x)\right\}
    \geq\min\left\{ c-\left\Vert g_{\rho}\right\Vert
    _{\infty},K_{0}\right\} \geq\xi.
  \]
  Note that we have proved that $c$ is a \emph{classical} supersolution
  of $F_{\rho}-\xi=0$.

  Without loss of generality, we assume that $w$ is lower semicontinuous
  (otherwise, replace $w$ by its lower semicontinuous envelope). Now,
  let $(t,x,\varphi_{\lambda})\in\mathcal{O}\times C^{1,2}(\mathcal{O})$
  be such that $(w_{\lambda}-\varphi_{\lambda})(t,x)=0$ is a local
  minimum of $w_{\lambda}-\varphi_{\lambda}$. Further letting
  $\lambda^{\prime}\coloneqq1-\lambda$
  for brevity and $\varphi\coloneqq(\varphi_{\lambda}-\lambda
  c)/\lambda^{\prime}$,
  it follows that $(t,x)$ is also a local minimum point of $w-\varphi$
  since
  \begin{align*}
    \lambda^{\prime}\left(w-\varphi\right) &
    =\lambda^{\prime}\left(w-\left(\varphi_{\lambda}-\lambda
    c\right)/\lambda^{\prime}\right)=\lambda^{\prime}w+\lambda
    c-\varphi_{\lambda}=w_{\lambda}-\varphi_{\lambda}.
  \end{align*}
  We now seek to show that
  \[
    -\inf_{b\in B}\left\{
    (\partial_{t}+\text{\L}^{b}-\rho)\varphi_{\lambda}(t,x)+f_{\rho}(t,x,b)\right\}
    \geq\lambda\xi,
  \]
  for which it is sufficient to show that for some choice of $b\in B$,
  \[
    (\partial_{t}+\text{\L}^{b}-\rho)\varphi_{\lambda}(t,x)+f_{\rho}(t,x,b)\leq-\lambda\xi.
  \]
\begin{comment}
By the supersolution property of $w$,
\[
-\inf_{b\in B}\left\{ \left(\partial_{t}+\text{\L}^{b}-\rho\right)\varphi(t,x)+f_{\rho}(t,x,b)\right\} \geq0,
\]
and hence
\[
\inf_{b\in B}\left\{ \left(\partial_{t}+\text{\L}^{b}-\rho\right)\varphi(t,x)+f_{\rho}(t,x,b)\right\} \leq0.
\]
Therefore, using compactness and continuity, there exists some $b\in B$
such that
\[
\left(\partial_{t}+\text{\L}^{b}-\rho\right)\varphi(t,x)+f_{\rho}(t,x,b)\leq0.
\]
Moreover, since
\[
-\sup_{b\in B}\left\{ -\rho c+f_{\rho}(t,x,b)\right\} \geq1,
\]
we have
\[
\sup_{b\in B}\left\{ -\rho c+f_{\rho}(t,x,b)\right\} \leq-1.
\]
Therefore, for all $b\in B$,
\[
-\rho c+f_{\rho}(t,x,b)\leq-1\leq-\xi.
\]
\end{comment}
  In particular, using the supersolution property of $w$ along with
  the continuity of $\varphi$ and compactness of $B$, there exists
  $b\in B$ such that%
\begin{comment}
\begin{align*}
0 & \geq\lambda^{\prime}\left(\left(\partial_{t}+\text{\L}^{b}-\rho\right)\varphi(t,x)+f_{\rho}(t,x,b)\right)\\
 & =\left(\partial_{t}+\text{\L}^{b}-\rho\right)\left(\varphi_{\lambda}(t,x)-\lambda c\right)+\lambda^{\prime}f_{\rho}(t,x,b)\\
 & =\left(\partial_{t}+\text{\L}^{b}-\rho\right)\varphi_{\lambda}(t,x)+f_{\rho}(t,x,b)-\lambda\left(-\rho c+f_{\rho}(t,x,b)\right)\\
 & \geq\left(\partial_{t}+\text{\L}^{b}-\rho\right)\varphi_{\lambda}(t,x)+f_{\rho}(t,x,b)+\lambda\xi.
\end{align*}
\end{comment}
  \[
    0\geq\lambda^{\prime}\left((\partial_{t}+\text{\L}^{b}-\rho)\varphi(t,x)+f_{\rho}(t,x,b)\right)\geq(\partial_{t}+\text{\L}^{b}-\rho)\varphi_{\lambda}(t,x)+f_{\rho}(t,x,b)+\lambda\xi.
  \]

  On $\operatorname{cl}\mathcal{O}$, since $w$ is a supersolution,
  we have that $w\geq\mathcal{M}_{\rho}w$. Along with the convexity
  of $\mathcal{M}_{\rho}$ (\prettyref{lem:more_intervention_results}),
  this yields
  \[
    w_{\lambda}-\mathcal{M}_{\rho}w_{\lambda}\geq
    w_{\lambda}-\lambda^{\prime}\mathcal{M}_{\rho}w-\lambda\mathcal{M}_{\rho}c\geq
    w_{\lambda}-\lambda^{\prime}w-\lambda\mathcal{M}_{\rho}c=\lambda(c-\mathcal{M}_{\rho}c)\geq\lambda\xi.
  \]
  Lastly, on $\partial^{+}\mathcal{O}$, we have
  \[
    w_{\lambda}-g_{\rho}=\lambda^{\prime}\left(w-g_{\rho}\right)+\lambda\left(c-g_{\rho}\right)\geq\lambda\xi,
  \]
  so that $w_{\lambda}$ satisfies the boundary condition.
\end{proof}
We give a result that describes the regularity of the ``non-impulse''
part of the HJBI-QVI \eqref{eq:hjbi_discounted}. In the lemma statement,
$I_{d}$ denotes the identity matrix in $\mathbb{R}^{d\times d}$.
\begin{lem}
  \label{lem:non_impulse_continuity}Let $H$ be given by
  \begin{equation}
    H(t,x,r,p,X)\coloneqq-\inf_{b\in B}\left\{
      \frac{1}{2}\operatorname{trace}(\sigma(x,b)\sigma^{\intercal}(x,b)X)+\left\langle
    \mu(x,b),p\right\rangle -\rho r+f_{\rho}(t,x,b)\right\}
    .\label{eq:hamiltonian}
  \end{equation}
  Then, there exists a positive constant $c$ such that for each compact
  set $D\subset\mathbb{R}^{d}$, there exists a modulus of continuity
  $\omega$ such that for all
  $(t,x,r,X),(s,y,r^{\prime},Y)\in[0,T]\times
  D\times\mathbb{R}\times\mathscr{S}(d)$
  satisfying
  \[
    \left(
      \begin{array}{cc}
        X\\
        & -Y
    \end{array}\right)\preceq3\alpha\left(
      \begin{array}{cc}
        I_{d} & -I_{d}\\
        -I_{d} & I_{d}
    \end{array}\right)
  \]
  and all positive constants $\alpha$ and $\epsilon$,
  \begin{multline*}
    H(s,y,r^{\prime},\alpha\left(x-y\right)-\epsilon y,Y-\epsilon
    I_{d})-H(t,x,r,\alpha\left(x-y\right)+\epsilon x,X+\epsilon I_{d})\\
    \leq\rho\left(r^{\prime}-r\right)+c\,(\alpha\,|x-y|^{2}+\epsilon\,(1+|x|^{2}+|y|^{2}))+\omega(\left|(t,x)-(s,y)\right|).
  \end{multline*}
\end{lem}
\begin{proof}
  Let $M(b)\coloneqq\sigma(x,b)$ and $N(b)\coloneqq\sigma(y,b)$. First,
  note that
  \begin{multline}
    H(s,y,r^{\prime},\alpha\left(x-y\right)-\epsilon y,Y-\epsilon
    I_{d})-H(t,x,r,\alpha\left(x-y\right)+\epsilon x,X+\epsilon I_{d})\\
    \leq\sup_{b\in B}\left\{
      \begin{gathered}\operatorname{trace}(M(b)M(b)^{\intercal}\left(X+\epsilon
        I_{d}\right)-N(b)N(b)^{\intercal}\left(Y-\epsilon I_{d}\right))/2\\
        +\alpha\left\langle \mu(x,b)-\mu(y,b),x-y\right\rangle
        +\epsilon\left\langle \mu(x,b),x\right\rangle
        +\epsilon\left\langle \mu(y,b),y\right\rangle \\
        +\rho\left(r^{\prime}-r\right)+f_{\rho}(t,x,b)-f_{\rho}(s,y,b)
      \end{gathered}
    \right\} .\label{eq:hamiltonian_difference}
  \end{multline}
  Omitting the dependence on $b$ for brevity and employing the linear
  growth of $\mu$ and the inequality %
\begin{comment}
$|x|\leq\max\{1,|x|\}\leq\max\{1,|x|^{2}\}\leq1+|x|^{2}$
\end{comment}
  $|x|\leq1+|x|^{2}$,
  \[
    \epsilon\left\langle \mu(x),x\right\rangle +\epsilon\left\langle
    \mu(y),y\right\rangle
    \leq\operatorname{const.}\epsilon\,((1+|x|)|x|+(1+|y|)|y|)\leq\operatorname{const.}\epsilon\,(1+|x|^{2}+|y|^{2}).
  \]
  Denoting by $\left\Vert \cdot\right\Vert _{F}$ the Frobenius norm,
  the linear growth of $\sigma$ similarly yields
  \[
    \epsilon\operatorname{trace}(MM^{\intercal})+\epsilon\operatorname{trace}(NN^{\intercal})=\epsilon\left\Vert
    M\right\Vert _{F}^{2}+\epsilon\left\Vert N\right\Vert
    _{F}^{2}\leq\operatorname{const.}\epsilon\,(1+|x|^{2}+|y|^{2}).
  \]
  We also have the inequalities
  \[
    \alpha\left\langle \mu(x)-\mu(y),x-y\right\rangle
    \leq\alpha\left|\mu(x)-\mu(y)\right|\left|x-y\right|
  \]
  and
  \begin{align*}
    \operatorname{trace}(MM^{\intercal}X-NN^{\intercal}Y) &
    =\operatorname{trace}\left(\left(
        \begin{array}{cc}
          MM^{\intercal} & MN^{\intercal}\\
          NM^{\intercal} & NN^{\intercal}
      \end{array}\right)\left(
        \begin{array}{cc}
          X\\
          & -Y
    \end{array}\right)\right)\\
    & \leq3\alpha\operatorname{trace}\left(\left(
        \begin{array}{cc}
          MM^{\intercal} & MN^{\intercal}\\
          NM^{\intercal} & NN^{\intercal}
      \end{array}\right)\left(
        \begin{array}{cc}
          I_{d} & -I_{d}\\
          -I_{d} & I_{d}
    \end{array}\right)\right)\\
    &
    =3\alpha\operatorname{trace}\left(\left(M-N\right)\left(M-N\right)^{\intercal}\right)\\
    & =3\alpha\left\Vert M-N\right\Vert _{F}^{2}.
  \end{align*}
  The desired result follows by applying the above inequalities to
  \eqref{eq:hamiltonian_difference}
  and invoking the uniform continuity of $f$ on the compact set
  $[0,T]\times D\times B$
  and the Lipschitzness of $\mu$ and $\sigma$.
\end{proof}
The following appears in \citep[Problem 2.4.17]{MR1751334}.
\begin{lem}
  \label{lem:limsup_product}Let $(a_{n})_{n}$ and $(b_{n})_{n}$ be
  sequences of nonnegative numbers. If $a_{n}$ converges to a positive
  number $a$,
  $\limsup_{n\rightarrow\infty}a_{n}b_{n}=a\limsup_{n\rightarrow\infty}b_{n}$.
\end{lem}
We are finally ready to prove the comparison principle. We mention
that  we cannot directly use the ``parabolic'' Crandall-Ishii lemma
\citep[Theorem 8.3]{MR1118699} in the proof since condition (8.5)
of \citep{MR1118699} cannot be satisfied due to the impulse term
(see also the proof of \citep[Lemma 11.1]{MR3268054} for a more in-depth
discussion of this issue). We rely instead on the ``elliptic'' Crandall-Ishii
lemma \citep[Theorem 3.2]{MR1118699} and employ a variable-doubling
argument inspired by \citep[Lemma 8]{MR1073054}. Below, we use the
parabolic semijets $\mathscr{P}_{\mathcal{O}}^{2,\pm}u(t,x)$ and
their closures $\operatorname{cl}(\mathscr{P}_{\mathcal{O}}^{2,\pm}u(t,x))$,
defined in \citep[\S 8]{MR1118699}.
\begin{proof}[Proof of \prettyref{thm:comparison_principle}]
  Let $u$ be a bounded subsolution and $w$ be a bounded supersolution
  of \eqref{eq:hjbi_discounted}. As in the proof of
  \prettyref{lem:strict_supersolutions},
  we can assume that $u$ (resp. $w$) is upper (resp. lower) semicontinuous
  (otherwise, replace $u$ and $w$ by their semicontinuous envelopes).
  Let $c$ be given as in \prettyref{lem:strict_supersolutions} and
  $w_{m}\coloneqq(1-1/m)w+c/m$ for all integers $m>1$. Note that
  \[
    \sup_{\operatorname{cl}\mathcal{O}}\left\{ u-w_{m}\right\}
    =\sup_{\operatorname{cl}\mathcal{O}}\left\{
    u-w+\left(w-c\right)/m\right\}
    \geq\sup_{\operatorname{cl}\mathcal{O}}\left\{ u-w\right\}
    -\left(\left\Vert w\right\Vert _{\infty}+c\right)/m.
  \]
  Therefore, to prove the comparison principle, it is sufficient to
  show $u-w_{m}\leq0$ (pointwise) along a subsequence of $(w_{m})_{m}$.
  We establish it for all $m$.

  To that end, fix $m$ and suppose
  $\delta\coloneqq\sup_{\operatorname{cl}\mathcal{O}}\{u-w_{m}\}>0$.
  Letting $\nu>0$, we can find
  $(t^{\nu},x^{\nu})\in\operatorname{cl}\mathcal{O}$ such
  that $(u-w_{m})(t^{\nu},x^{\nu})\geq\delta-\nu$. Let
  \[
    \varphi(t,x,s,y)\coloneqq\frac{\alpha}{2}\left(\left|t-s\right|^{2}+\left|x-y\right|^{2}\right)+\frac{\epsilon}{2}\left(\left|x\right|^{2}+\left|y\right|^{2}\right)
  \]
  be a smooth function parameterized by constants $\alpha>0$ and
  $0<\epsilon\leq1$.
  Further let $\Phi(t,x,s,y)\coloneqq u(t,x)-w_{m}(s,y)-\varphi(t,x,s,y)$
  and note that
  \begin{align*}
    \sup_{(t,x,s,y)\in([0,T]\times\mathbb{R}^{d})^{2}}\Phi(t,x,s,y) &
    \geq\sup_{(t,x)\in[0,T]\times\mathbb{R}^{d}}\left\{
    (u-w_{m})(t,x)-\epsilon\left|x\right|^{2}\right\} \\
    & \geq(u-w_{m})(t^{\nu},x^{\nu})-\epsilon|x^{\nu}|^{2}\\
    & \geq\delta-\nu-\epsilon|x^{\nu}|^{2}.
  \end{align*}
  We henceforth assume $\nu$ and $\epsilon$ are small enough (e.g.,
    pick $\nu\leq\delta/4$ and then $\epsilon>0$ such that
  $\epsilon|x^{\nu}|^{2}\leq\delta/4$)
  to ensure that $\delta-\nu-\epsilon|x^{\nu}|^{2}$ is positive.

  Since $u$ and $w_{m}$ are bounded (and thus trivially of subquadratic
  growth), it follows that $\Phi$ admits a maximum at
  $(t_{\alpha},x_{\alpha},s_{\alpha},y_{\alpha})\in([0,T]\times\mathbb{R}^{d})^{2}$
  such that
  \begin{equation}
    \left\Vert u\right\Vert _{\infty}+\left\Vert w_{m}\right\Vert
    _{\infty}\geq
    u(t_{\alpha},x_{\alpha})-w_{m}(s_{\alpha},y_{\alpha})\geq\delta-\nu-\epsilon|x^{\nu}|^{2}+\varphi(t_{\alpha},x_{\alpha},s_{\alpha},y_{\alpha}).\label{eq:comparison_principle_proof_inequality_0}
  \end{equation}
  Since $-\epsilon|x^{\nu}|^{2}\geq-|x^{\nu}|^{2}$, the above inequality
  implies that
  \[
    \alpha\left(|t_{\alpha}-s_{\alpha}|^{2}+|x_{\alpha}-y_{\alpha}|^{2}\right)+\epsilon\left(|x_{\alpha}|^{2}+|y_{\alpha}|^{2}\right)
  \]
  is bounded independently of $\alpha>0$ and $0<\epsilon\leq1$ (but
  not of $\nu$ since $|x^{\nu}|$ may be arbitrarily large).

  Now, for fixed $\epsilon$, consider some sequence of increasing $\alpha$,
  say $(\alpha_{n})_{n}$, such that $\alpha_{n}\rightarrow\infty$.
  To each $\alpha_{n}$ is associated a maximum point
  $(t_{n},x_{n},s_{n},y_{n})\coloneqq(t_{\alpha_{n}},x_{\alpha_{n}},s_{\alpha_{n}},y_{\alpha_{n}})$.
  By the discussion above, $\{(t_{n},x_{n},s_{n},y_{n})\}_{n}$ is contained
  in a compact set. Therefore, $(\alpha_{n},t_{n},x_{n},s_{n},y_{n})_{n}$
  admits a subsequence whose four last components converge to some point
  $(\hat{t},\hat{x},\hat{s},\hat{y})$. With a slight abuse of notation,
  we relabel this subsequence $(\alpha_{n},t_{n},x_{n},s_{n},y_{n})_{n}$,
  forgetting the original sequence. It follows that $\hat{x}=\hat{y}$
  since otherwise $|\hat{x}-\hat{y}|>0$ and \prettyref{lem:limsup_product}
  implies
  \[
    \limsup_{n\rightarrow\infty}\left\{
    \alpha_{n}|x_{n}-y_{n}|^{2}\right\}
    =\limsup_{n\rightarrow\infty}\alpha_{n}\left|\hat{x}-\hat{y}\right|^{2}=\infty,
  \]
  contradicting the boundedness in the discussion above. The same exact
  argument yields $\hat{t}=\hat{s}$. Moreover, letting
  $\varphi_{n}\coloneqq\varphi(t_{n},x_{n},s_{n},y_{n};\alpha_{n})$,
  \begin{align}
    0\leq\limsup_{n\rightarrow\infty}\varphi_{n} &
    \leq\limsup_{n\rightarrow\infty}\left\{
    u(t_{n},x_{n})-w_{m}(s_{n},y_{n})\right\}
    -\delta+\nu+\epsilon|x^{\nu}|^{2}\nonumber \\
    &
    \leq(u-w_{m})(\hat{t},\hat{x})-\delta+\nu+\epsilon|x^{\nu}|^{2}\label{eq:comparison_principle_proof_inequality_1}
  \end{align}
  and hence
  \begin{equation}
    0<\delta-\nu-\epsilon|x^{\nu}|^{2}\leq(u-w_{m})(\hat{t},\hat{x}).\label{eq:comparison_principle_proof_inequality_2}
  \end{equation}

  By \prettyref{lem:strict_supersolutions},
  $(w_{m}-\mathcal{M}_{\rho}w_{m})(s_{n},y_{n})\geq\xi/m$.
  Suppose, in order to arrive at a contradiction,
  $(\alpha_{n},t_{n},x_{n},s_{n},y_{n})_{n}$
  admits a subsequence along which $(u-\mathcal{M}_{\rho}u)(t_{n},x_{n})\leq0$.
  As usual, we abuse slightly the notation and temporarily refer to
  this subsequence as $(\alpha_{n},t_{n},x_{n},s_{n},y_{n})_{n}$. Combining
  these two inequalities,
  \begin{align*}
    -\xi/m & \geq
    u(t_{n},x_{n})-w_{m}(s_{n},y_{n})-\left(\mathcal{M}_{\rho}u(t_{n},x_{n})-\mathcal{M}_{\rho}w_{m}(s_{n},y_{n})\right)\\
    &
    \geq\delta-\nu-\epsilon|x^{\nu}|^{2}+\mathcal{M}_{\rho}w_{m}(s_{n},y_{n})-\mathcal{M}_{\rho}u(t_{n},x_{n}).
  \end{align*}
  Taking limit inferiors with respect to $n\rightarrow\infty$ of both
  sides of this inequality and using the semicontinuity established
  in \prettyref{lem:intervention_results} yields
  \[
    -\xi/m\geq\delta-\nu-\epsilon|x^{\nu}|^{2}+\mathcal{M}_{\rho}w_{m}(\hat{t},\hat{x})-\mathcal{M}_{\rho}u(\hat{t},\hat{x}).
  \]
  It follows, by the upper semicontinuity of $u$, that the supremum
  in $\mathcal{M}_{\rho}u(\hat{t},\hat{x})$ is achieved at some $\hat{z}\in Z$.
  Therefore,
  \[
    -\xi/m\geq\delta-\nu-\epsilon|x^{\nu}|^{2}+w_{m}(\hat{t},\hat{x}+\Gamma(\hat{t},\hat{z}))-u(\hat{t},\hat{x}+\Gamma(\hat{t},\hat{z}))\geq-\nu-\epsilon|x^{\nu}|^{2}.
  \]
  Taking $\nu$ and $\epsilon$ small enough %
\begin{comment}
(e.g., pick $\nu\leq\xi/(4m)$ and $\epsilon\leq\xi/(4m|x^{\nu}|^{2})$)
\end{comment}
  {} yields a contradiction. By virtue of the above, we may assume that
  our original sequence $(\alpha_{n},t_{n},x_{n},s_{n},y_{n})_{n}$
  whose four last components converge to $(\hat{t},\hat{x},\hat{s},\hat{y})$
  satisfies $(u-\mathcal{M}_{\rho}u)(t_{n},x_{n})>0$ for all $n$.

  Now, suppose $\hat{t}=T$. By \prettyref{lem:strict_supersolutions},
  $(w_{m}-\mathcal{M}_{\rho}w_{m})(T,\hat{x})\geq\xi/m$ and
  $w_{m}(T,\hat{x})-g_{\rho}(\hat{x})\geq0$.
  If $(u-\mathcal{M}_{\rho}u)(T,\hat{x})\leq0$, we arrive at a contradiction
  by an argument similar to the above. It follows that
  $u(T,\hat{x})-g_{\rho}(\hat{x})\leq0$
  and hence $(u-w_{m})(T,\hat{x})\leq0$, contradicting
  \eqref{eq:comparison_principle_proof_inequality_2}.
  We conclude that $\hat{t}<T$ so that we may safely assume
  $(t_{n},x_{n},s_{n},y_{n})\in\mathcal{O}$
  for all $n$.

  We are now in a position to apply the Crandall-Ishii lemma
  \citep[Theorem 3.2]{MR1118699},
  which implies the existence of $X_{n},Y_{n}\in\mathscr{S}(d)$
  satisfying\footnote{The elliptic Crandall-Ishii actually gives us
    $(X_{n},Y_{n})\in\mathscr{S}(d+1)$.
    An argument using the fact that \emph{the principal submatrices of
    a positive semidefinite (PSD) matrix are PSD} allows us to discard
  the extra dimension associated with time.}
  \begin{align*}
    (\partial_{t}\varphi_{n},D_{x}\varphi_{n},X_{n}+\epsilon I_{d}) &
    \in\operatorname{cl}(\mathscr{P}_{\mathcal{O}}^{2,+}u(t_{n},x_{n})),\\
    (-\partial_{s}\varphi_{n},-D_{y}\varphi_{n},Y_{n}-\epsilon I_{d})
    & \in\operatorname{cl}(\mathscr{P}_{\mathcal{O}}^{2,-}w_{m}(s_{n},y_{n})),
  \end{align*}
  and
  \[
    -3\alpha_{n}I_{2d}\preceq\left(
      \begin{array}{cc}
        X_{n}\\
        & -Y_{n}
    \end{array}\right)\preceq3\alpha_{n}\left(
      \begin{array}{cc}
        I_{d} & -I_{d}\\
        -I_{d} & I_{d}
    \end{array}\right).
  \]
  Due to our choice of $\varphi$, we get
  \begin{align*}
    a_{n}\coloneqq\partial_{t}\varphi_{n}=\partial_{t}\varphi(t_{n},x_{n},s_{n},y_{n};\alpha_{n})
    & =\alpha_{n}\left(t_{n}-s_{n}\right)\\
    &
    =-\partial_{s}\varphi(t_{n},x_{n},s_{n},y_{n};\alpha_{n})=-\partial_{s}\varphi_{n}
  \end{align*}
  along with
  \[
    D_{x}\varphi_{n}=\alpha_{n}(x_{n}-y_{n})+\epsilon x_{n}\text{ and
    }D_{y}\varphi_{n}=-\alpha_{n}(x_{n}-y_{n})+\epsilon y_{n}.
  \]
  Therefore, since $(u-\mathcal{M}_{\rho}u)(t_{n},x_{n})>0$,
  \prettyref{lem:can_use_semijets} of \prettyref{app:equivalent_definitions}
  yields
  \begin{align}
    -a_{n}+H(t_{n},x_{n},u(t_{n},x_{n}),\alpha_{n}(x_{n}-y_{n})+\epsilon
    x_{n},X_{n}+\epsilon I_{d}) & \leq0\nonumber \\
    \text{ and
    }-a_{n}+H(s_{n},y_{n},w_{m}(s_{n},y_{n}),\alpha_{n}(x_{n}-y_{n})-\epsilon
    y_{n},Y_{n}-\epsilon I_{d}) &
    \geq0.\label{eq:comparison_principle_proof_inequality_3}
  \end{align}
  We can combine the inequalities
  \eqref{eq:comparison_principle_proof_inequality_3}
  and apply \prettyref{lem:non_impulse_continuity} to get
  \begin{align}
    0 & \leq
    H(s_{n},y_{n},w_{m}(s_{n},y_{n}),\alpha_{n}(x_{n}-y_{n})-\epsilon
    y_{n},Y_{n}-\epsilon I_{d})\nonumber \\
    &
    \qquad-H(t_{n},x_{n},u(t_{n},x_{n}),\alpha_{n}(x_{n}-y_{n})+\epsilon
    x_{n},X_{n}+\epsilon I_{d})\nonumber \\
    &
    \leq\rho\left(w_{m}(s_{n},y_{n})-u(t_{n},x_{n})\right)+c\,(\alpha_{n}|x_{n}-y_{n}|^{2}+\epsilon(1+|x_{n}|^{2}+|y_{n}|^{2}))\nonumber
    \\
    & \qquad+\omega(\left|(t_{n},x_{n})-(s_{n},y_{n})\right|)\nonumber \\
    &
    \leq\rho\left(w_{m}(s_{n},y_{n})-u(t_{n},x_{n})\right)+2c\left(\varphi_{n}+\epsilon\right)+\omega(\left|(t_{n},x_{n})-(s_{n},y_{n})\right|)\label{eq:comparison_principle_proof_inequality_4}
  \end{align}
  where $\omega$ is a modulus of continuity. Moreover, by
  \eqref{eq:comparison_principle_proof_inequality_0},
  \begin{equation}
    w_{m}(s_{n},y_{n})-u(t_{n},x_{n})\leq-\delta+\nu+\epsilon|x^{\nu}|^{2},\label{eq:comparison_principle_proof_inequality_5}
  \end{equation}
  and by \eqref{eq:comparison_principle_proof_inequality_1},
  \begin{equation}
    \limsup_{n\rightarrow\infty}\varphi_{n}\leq\nu+\epsilon|x^{\nu}|^{2}.\label{eq:comparison_principle_proof_inequality_6}
  \end{equation}
  Applying \eqref{eq:comparison_principle_proof_inequality_5} to
  \eqref{eq:comparison_principle_proof_inequality_4},
  taking the limit superior as $n\rightarrow\infty$ of both sides,
  and finally applying \eqref{eq:comparison_principle_proof_inequality_6}
  to the resulting expression yields%
\begin{comment}
\begin{align*}
0 & \leq\rho\left(w_{m}(s_{n},y_{n})-u(t_{n},x_{n})\right)+c\left(2\varphi_{n}+\epsilon\right)+\omega(\left|(t_{n},x_{n})-(s_{n},y_{n})\right|)\\
 & \leq\rho\left(-\delta+\nu+\epsilon|x^{\nu}|^{2}\right)+2c\left(\nu+\epsilon|x^{\nu}|^{2}+\epsilon\right).
\end{align*}
\end{comment}
  \[
    \delta\leq\operatorname{const.}\left(\nu+\epsilon+\epsilon|x^{\nu}|^{2}\right)
  \]
  ($\operatorname{const.}$ above depends on $\rho$ and $c$). Picking
  $\nu$ small enough and taking $\epsilon\rightarrow0$ yields the
  desired contradiction.
\end{proof}

\section{\label{sec:conclusion}Extensions and future work}

By adapting the technique in \citep[Theorem 5.11]{MR2568293}, one
should be able to extend the comparison principle to solutions of
arbitrary polynomial growth. However, for polynomial degree $p$ growth,
the resulting uniqueness theorem requires the existence of a \emph{classical}
``strict'' supersolution $c\coloneqq c(t,x)$ (similar to
\prettyref{lem:strict_supersolutions})
satisfying $c(t,x)/|x|^{p}\rightarrow\infty$ as
$|x|\rightarrow\infty$ (uniformly in $t$).
Unfortunately, the construction of such a solution is ad hoc (i.e.,
problem dependent); see the discussion at the end of \citep[Section
2.5]{MR2568293}.

We also mention here a trivial but relevant extension: removing altogether
the requirement $K(t,z)\in\omega(1)$ of \prettyref{assu:impulse_assumptions}
\ref{enu:negative_growth_condition} and redefining $Z\coloneqq Z(t,x)$,
$\Gamma\coloneqq\Gamma(t,x,z)$, and $K\coloneqq K(t,x,z)$ to depend
on $x$ so that
\[
  \mathcal{M}u(t,x)\coloneqq\sup_{z\in Z(t,x)}\left\{
  u(t,x+\Gamma(t,x,z))+K(t,x,z)\right\}
\]
does not invalidate the comparison principle if we require $Z(t,x)$
to be compact for each $(t,x)\in\operatorname{cl}\mathcal{O}$ and
$(t,x)\mapsto Z(t,x)$ to be continuous (in the Hausdorff metric)
in order for Lemmas \ref{lem:intervention_results} and
\ref{lem:more_intervention_results}
to remain valid (see \citep[Lemma 5.1]{MR2284012} and \citep[Lemma
4.3]{MR2568293}).
This ``state-dependent'' setting is important namely because it
appears in practical impulse control problems \citep[Section 6]{MR3493959}.
However, the reader will find (with some reflection) that defining
$v^{+}$ and $v^{-}$ when $\Gamma$, $K$, and $Z$ depend on $x$
is a nontrivial matter. As such, this is an interesting direction
for future work.

Another possible extension would be to establish the value of a game
in which both players employ impulse (and stochastic) controls. This
may be a nontrivial undertaking if we wish to have unrestricted cost
functions for both players.  Since most of the issues in this work
arise from the dynamic programming approach, a natural way to tackle
this setting may be to use instead stochastic Perron's method
\citep{MR3162260,MR3206980}.
We suspect that using stochastic Perron's method may also allow one
to weaken some assumptions. We refer the interested reader to \citep{MR3553921},
in which stochastic Perron's method is applied to a switching game
setting that is similar to our own.

\begin{comment}
A natural future direction for this work is to extend the findings
to Lévy SDE. Since no assumptions of path-wise continuity of the process
were needed in the proof of the DPP, the main challenges faced are
extending the DPE and comparison principle. A possible methodology
for extending the DPE would be to use the approximate début-type stopping
times along with the stochastic continuity of the process. We refer
to \citep{MR2422079} for reading on comparison principles for second-order
parabolic integro-differential equations.
\end{comment}

\appendix

\section{\label{app:equivalent_definitions}Alternate characterizations of
viscosity solutions}

Since the DPP (\prettyref{thm:dpp}) holds only for stopping times
taking finitely many values, we are unable to use début-type stopping
times in order to derive the DPE (\prettyref{thm:dpe}). In this case,
to apply Dynkin's formula, we require our test functions to be compactly
supported. The following result affords us this luxury:
\begin{lem}
  \label{lem:compact_test_functions}Let \prettyref{def:viscosity_solution}$_{c}$
  refer to \prettyref{def:viscosity_solution} with $C^{1,2}(\mathcal{O})$
  replaced by $C_{c}^{1,2}(\mathcal{O})$. A subsolution (resp. supersolution)
  under \prettyref{def:viscosity_solution} is a subsolution (resp.
  supersolution) under \prettyref{def:viscosity_solution}$_{c}$ and
  vice versa.
\end{lem}
\begin{proof}
  One direction is trivial, since $C_{c}^{1,2}(\mathcal{O})\subset
  C^{1,2}(\mathcal{O})$.

  Suppose $u$ is a subsolution under \prettyref{def:viscosity_solution}$_{c}$
  and let $(t,x,\varphi)\in\mathcal{O}\times C^{1,2}(\mathcal{O})$
  be given as in \prettyref{def:viscosity_solution} (we need not consider
    the parabolic boundary $\partial^{+}\mathcal{O}$, as it is not ``tested''
  by $\varphi$). Since $t<T$, choose $\theta\in C_{c}^{1}([0,T))$
  with $\theta=1$ on a neighborhood of $t$, and choose $\xi\in
  C_{c}^{2}(\mathbb{R}^{d})$
  with $\xi=1$ on a neighborhood of $x$. Set $\psi\coloneqq\chi\varphi$ where
  \[
    \chi(s,y)\coloneqq\theta(s)\xi(y).
  \]
  Then $\psi\in C_{c}^{1,2}(\mathcal{O})$ and inherits all the local
  properties of $\varphi$ at $(t,x)$ since the two functions coincide
  on an open neighborhood. Therefore,
  \[
    F(t,x,u^{*},D\varphi(t,x),D^{2}\varphi(t,x))=F(t,x,u^{*},D\psi(t,x),D^{2}\psi(t,x))\leq0,
  \]
  as desired. The supersolution case is identical.
\end{proof}
Below, we give one direction of a characterization of viscosity solutions
to \eqref{eq:hjbi_discounted} using the closed parabolic semijets
$\operatorname{cl}(\mathscr{P}_{\mathcal{O}}^{2,\pm}u(t,x))$ (a natural
converse can also be established, but is not needed here). This result
is required in the proof of the comparison principle
(\prettyref{thm:comparison_principle}).
\begin{lem}
  \label{lem:can_use_semijets}If $u$ is a bounded upper (resp. lower)
  semicontinuous
  subsolution (resp. supersolution) of \eqref{eq:hjbi_discounted},
  then for all $(t,x)\in\mathcal{O}$ and
  $(a,p,X)\in\operatorname{cl}(\mathscr{P}_{\mathcal{O}}^{2,+}u(t,x))$
  (resp. $\operatorname{cl}(\mathscr{P}_{\mathcal{O}}^{2,-}u(t,x))$),
  \[
    F_{\rho}(t,x,u,(a,p),X)\leq0\text{ (resp. }\geq0\text{)}.
  \]
\end{lem}
We prove only the subsolution case; the supersolution case is symmetric.
\begin{proof}
  Suppose $u$ is an upper semicontinuous subsolution, $(t,x)\in\mathcal{O}$,
  and $(a,p,X)\in\operatorname{cl}(\mathscr{P}_{\mathcal{O}}^{2,+}u(t,x))$.
  By definition, we can find a
  $\mathcal{O}\times\mathbb{R}\times\mathbb{R}^{d+1}\times\mathscr{S}(d)$-valued
  sequence $(t_{n},x_{n},u(t_{n},x_{n}),a_{n},p_{n},X_{n})_{n}$ converging
  to $(t,x,u(t,x),a,p,X)$ such that
  $(a_{n},p_{n},X_{n})\in\mathscr{P}_{\mathcal{O}}^{2,+}u(t_{n},x_{n})$
  for all $n$. Since $\mathcal{M}_{\rho}u$ is upper semicontinuous by
  \prettyref{lem:intervention_results},
  the function defined by
  \[
    Q(t,x,r,a,p,X)\coloneqq\min\left\{
    -a+H(t,x,r,p,X),r-\mathcal{M}_{\rho}u(t,x)\right\}
  \]
  is lower semicontinuous. Therefore,
  \begin{multline*}
    0\geq\liminf_{n\rightarrow\infty}F_{\rho}(t_{n},x_{n},u,(a_{n},p_{n}),X_{n})=\liminf_{n\rightarrow\infty}Q(t_{n},x_{n},u(t_{n},x_{n}),a_{n},p_{n},X_{n})\\
    \geq Q(t,x,u(t,x),a,p,X)=F_{\rho}(t,x,u,(a,p),X).\qedhere
  \end{multline*}
\end{proof}

\bibliographystyle{plainnat}
\bibliography{main}

\end{document}